\documentclass[12pt]{article}
\usepackage[T1]{fontenc}
\usepackage[utf8]{inputenc}
\usepackage{amsthm}
\usepackage{amssymb, url}
\usepackage{amsmath}
\usepackage{amsfonts}
\usepackage{graphicx}
\usepackage{xcolor}
\definecolor{darkgreen}{rgb}{0.0, 0.6, 0.0}
\definecolor{revcolor}{rgb}{0.60, 0.00, 0.50}

\usepackage{float}
\usepackage{multicol}
\usepackage{subcaption}
\usepackage{multirow}
\usepackage[margin=1in]{geometry}
\usepackage{comment}
\newcommand{\tn}[1]{|\|#1\||}
\newtheorem{theorem}{Theorem}[section]
\newtheorem{lemma}[theorem]{Lemma}
\newtheorem{corollary}[theorem]{Corollary}

\newtheorem{remark}[theorem]{Remark}

\title{Numerical analysis of data assimilation for slightly compressible flow}
\author{Aytekin \c{C}{\i}b{\i}k\thanks{Department of Mathematics, Gazi University, Ankara, T\"urkiye}, \and Rui Fang\thanks{Department of Mathematics, The Ohio State University, Columbus, OH, 43210, USA (fang.1211@osu.edu). Corresponding author: Rui Fang.}}

\date{\today}
\begin{document}
\maketitle

\begin{center}
  \textit{This work is dedicated to Professor William Layton on the occasion of his 70th birthday.}
\end{center}

\begin{abstract}
Continuous data assimilation improves flow predictions by continually nudging a model toward available observational data. For slightly compressible flow, a recent model addresses the limitations of velocity-only nudging by assimilating both velocity and pressure data and nudging both quantities into the incompressible Navier--Stokes equations \cite{CibikFang26}; continuous-in-time error estimates and preliminary experiments show that this joint nudging is effective and substantially reduces the model error relative to velocity-only nudging. Motivated by these results, we carry out the numerical analysis of the model and its finite element discretizations. We establish stability and error estimates for the semi-discrete scheme and for the fully discrete, linearized backward Euler scheme. The analysis shows an infinite predictability horizon: the effect of the initial error decays exponentially in time, and the model error is first order in the observation resolution \(H\) and of order \(\mu_1^{-1/2}\) in the pressure nudging parameter $\mu_1$. Balancing these two error terms, we choose \(\mu_1=\mathcal{O}(H^{-2})\), which yields the optimal convergence rate. Numerical experiments confirm the predicted rates.
\end{abstract}

% REQUIRED
\noindent\textbf{Keywords.} 
continuous data assimilation, nudging, predictability, Navier--Stokes equations, slightly compressible flow.

\section{Introduction} 
Accurate prediction of fluid flow is limited by uncertainty in both the initial state
and the governing model. Continuous data assimilation (CDA) combines observational data with mathematical models to improve the accuracy of numerical predictions of fluid flows. In 2014, Azouani, Olson, and Titi \cite{AOT14} introduced a simple CDA algorithm based on nudging for incorporating observational data into mathematical models. Since then, CDA has been  studied for a wide range of fluid flow models, with most existing works focusing on the incompressible Navier--Stokes equations (NSE).

We consider an incompressible viscous fluid flow in a domain $\Omega \subset \mathbb{R}^d$, with $d=2$ or $3$, governed by NSE:
\begin{align}
u_{t}+u\cdot \nabla u-\nu \triangle u+\nabla p &= f(x),\quad
\nabla \cdot u=0,\quad \text{in }\Omega,\ 0<t\leq T, \\
u &= 0\ \text{ on }\partial \Omega,\quad u(x,0)=u_{0}(x).
\end{align}

Due to the nonlinear nature of the NSE, even small uncertainties in the initial condition can grow over time, limiting the predictability of long-term simulations \cite{Lorenz1963}. This motivates CDA, which incorporates observational data into the governing equations to improve prediction accuracy.

We denote the coarse velocity observations by
\begin{equation*}
u_{\text{obs}}(x,t)=I_{H}u(x,t), 
\end{equation*}
where $I_H$ denotes an observation (interpolation) operator associated with spatial resolution $H$. In the CDA framework, observational data are incorporated into the governing equations through nudging terms that drive the assimilated solution toward the observed state.

Nudging for data assimilation originates from the pioneering work of Luenberger \cite{Luenberger1964} and has been extensively applied in geophysical flow simulations; see, for example, Kalnay \cite{kalnay2003atmospheric} and Navon \cite{NAVON199855}. Since the work of Azouani, Olson, and Titi \cite{AOT14}, CDA has been extensively analyzed for a wide range of fluid flow models, including the three-dimensional NSE studied by Biswas and Price \cite{biswas2021continuous}. Larios, Rebholz, and Zerfas \cite{larios2019global} established the global-in-time stability and accuracy of a fully discrete finite element CDA algorithm; {\c{C}}{\i}b{\i}k, Fang, Layton, and Siddiqua \cite{ccibik2025adaptive} developed a self-adaptive nudging parameter selection algorithm. Fang and Pakzad \cite{fang2026globalrecoverylocaldata} showed that local nudging can recover global information. Cao, Giorgini, Jolly, and Pakzad \cite{cao2022continuous} studied the three-dimensional Ladyzhenskaya model. A two-step modular nudging algorithm was later developed in \cite{CFL26}, allowing existing numerical codes to be used with minimal modification, and Yang \cite{yang2026higheraccuracymodulardata} subsequently extended the method to second-order accuracy.

More recently, CDA with model mismatch has received increasing attention. A recent work of Feireisl, Gwiazda, and \'{S}wierczewska-Gwiazda \cite{feireisl2026} applied standard velocity nudging to a reduced rotating Oberbeck--Boussinesq model using observations generated from the compressible primitive equations. In \cite{CIBIK2026172}, a model mismatch of the form $\omega R(u)$, where $R$ is a Lipschitz continuous operator, was considered, and the resulting model error was shown to decay at the rate $\mathcal{O}(\chi^{-1/2})$, where $\chi$ is the velocity nudging parameter. Carlson, Hudson, and Larios considered the discrepancy between the approximate and true viscosity in the two-dimensional NSE \cite{Carlson_Hudson_Larios_2020}. Clark Di Leoni, Mazzino, and Biferale \cite{PhysRevFluids.3.104604} used CDA to infer unknown physical parameters in turbulent flows, including the rotation rate associated with the Coriolis force.

Although the incompressible NSE are widely used in practice, real fluid flows generally exhibit slight compressibility \cite{oran2002fluid}. For \emph{slightly compressible} flows at low Mach numbers, under the Stokes hypothesis, the slightly compressible NSE are given by
\begin{align}
u_t +u \cdot \nabla u +\tfrac{1}{2} (\nabla \cdot u) u
    -\nu \Delta u -\tfrac{1}{3}\nu \nabla(\nabla \cdot u) +\nabla p
    &= f(x), \label{compressible_momentum}\\
\frac{1}{c^2}\!\left(\frac{\partial p}{\partial t} + u\cdot \nabla p\right)
    +\nabla \cdot u &= 0, \label{compressible_continuity}
\end{align}
where $c$ is the isentropic speed of sound.

The continuity equation couples pressure and velocity. The sound speed $c$ characterizes the compressibility of the system. As $c\to\infty$, the model formally reduces to the incompressible NSE.

Motivated by this setting, \cite{CibikFang26} proposed a CDA model that assimilates data from the slightly compressible NSE into the incompressible NSE through nudging both velocity and pressure. This paper establishes the numerical analysis of this formulation.

\subsection{Algorithm}
We recall the velocity-pressure nudging CDA algorithm introduced in \cite{CibikFang26}, which serves as the basis for the present analysis. The algorithm assimilates data from a slightly compressible reference flow, stated in (\ref{compressible_momentum})--(\ref{compressible_continuity}), into an incompressible Navier--Stokes model by simultaneously correcting the velocity and pressure fields toward the observations.
The assimilated system reads
\begin{align}
v_t +v \cdot \nabla v -\nu \Delta v +\nabla q
    - \chi I_H(u-v) &= f(x), \label{momentum_eqn_v1}\\
\nabla \cdot v - \mu_1 I_H(p-q) - \mu_2(I_H(q)-q) &= 0,
    \label{continuity_eqn_v1}
\end{align}
where \(\chi\) is the velocity nudging parameter, and \(\mu_1,\mu_2\) are pressure nudging parameters. The parameter \(\mu_1\) assimilates the observed pressure data, while \(\mu_2\) regularizes unresolved pressure modes through \(q-I_H(q)\).

The analysis relies on two conditions: the observation density is sufficiently high ($H$ is sufficiently small) and the velocity nudging parameter $\chi$ is sufficiently large \cite{CFL26}. Specifically, for the continuous analysis, we choose $H$ and $\chi$ so that the following conditions hold:
\begin{eqnarray}\label{condi_3d}
\chi-\frac{27}{16}\frac{C_2^4}{\nu^3}\|\nabla v\|_{L^{\infty
}(0,T;L^{2}(\Omega ))}^{4} &\geq &\alpha \chi >0, \\
\nu -2\left( C_{1}H\right)^{2}\chi  &>&0, \label{h_condi}
\end{eqnarray}
for some fixed $\alpha \in (0,1)$.

The remainder of the paper is organized as follows. Section 2 establishes stability estimates for the continuous problem, the semi-discrete scheme, and the fully discrete linearized backward Euler and implicit-explicit second-order backward differentiation formula (BDF2-IMEX) schemes. Section 3 establishes the semi-discrete error estimates and the fully discrete error estimates for the linearized backward Euler scheme. Section 4 presents numerical experiments, and Section 5 concludes the paper.
 
\subsection{Notations and Preliminaries}
We denote the $L^{2}(\Omega )$ norm and inner product by $\Vert \cdot \Vert $ and $(\cdot , \cdot)$, respectively. By $\Vert \cdot \Vert_{L^{p}}$ we indicate the $L^{p}(\Omega )$ norm, and \(|\cdot|_{m+1}\) denotes the \(H^{m+1}(\Omega)\) seminorm. The solution spaces $X$ for the velocity and $Q$ for the pressure are defined as:
\begin{equation*}
\begin{aligned}
    &X:=(H_0^1(\Omega))^d=\{ v\in (L^2(\Omega))^d: \nabla  v\in (L^2(\Omega))^{d \times d}\ \text{and}\  v=0\ \text{on}\ \partial\Omega\},\\
    &Q:=L^2_0(\Omega)=\{q\in L^2(\Omega): \int_\Omega q\ d x=0\}.
\end{aligned}
\end{equation*}
The dual norm on \(X'=(H^{-1}(\Omega))^d\) is defined by
\begin{equation*}
\|f\|_{-1}
:=
\sup_{0\neq v\in X}
\frac{\langle f,v\rangle_{X',X}}{\|\nabla v\|},
\qquad f\in X'.
\end{equation*}
When \(f\in (L^2(\Omega))^d\), the duality pairing reduces to the
\(L^2(\Omega)\) inner product:
\[
\langle f,v\rangle_{X',X}=(f,v).
\]

We define discrete time notations as follows. The time step is denoted by $\Delta t >0$ and $t_n = n\Delta t, n = 0,1,\cdots,N = \frac{T}{\Delta t}$. Given a Banach space $Y$, we define the following norms:
\begin{equation*}
		|\| v \||_{l^{p} \left( Y \right) } := \left(\Delta t\ \sum_{n=0}^{N}\|v^n  \|_Y^p \right)^{1/p}\quad \text{and} \quad |\| v \||_{l^{\infty}\left(Y\right)} := \max_{0\leq n \leq N } \| v^n \|_Y .
\end{equation*}

The finite element method for this problem involves picking finite element spaces \cite{L08}, $X^h\subset X$ and $Q^h\subset Q$. We assume that $(X^h, Q^h)$ are conforming and have the following approximation properties: for $u\in \left(H^{m+1}(\Omega)\right)^d\ \cap\ \left(H_0^1(\Omega)\right)^d$ and $p\in H^{m}(\Omega)$,
\begin{equation}\label{prop}
\begin{aligned}
\inf_{v^h\in X^h}\{\|u-v^h\|+h\|\nabla(u-v^h)\|\}&\leq Ch^{m+1}|u|_{m+1},
\\\inf_{q^h\in Q^h}\|p-q^h\|&\leq Ch^{m}|p|_{m},
\end{aligned}
\end{equation}
\begin{equation}\label{infsup}
\inf_{q^h\in Q^h}\sup_{ v^h\in X^h}\frac{(q^h,\nabla\cdot  v^h)}{\|q^h\|\|\nabla  v^h\|}\geq \gamma_h>0.
\end{equation}
Here $\gamma_h > 0$ denotes the discrete inf–sup constant.
We use the following inequalities. The first estimate of this type was derived in 1958 by Nash \cite{N58}. The improvements to some constants are mentioned in \cite{CFL26}.

\begin{lemma}\label{norms_property} For any vector function $u:{{{{\mathbb{R}}}}}%
^{d}\rightarrow {{{{\mathbb{R}}}}}^{d}$ with compact support and with the
indicated $L^{p}$ norms finite, 
\begin{align*}
\Vert u\Vert _{L^{4}({{{{\mathbb{R}}}}}^{2})}& \leq 2^{1/4}\Vert u\Vert^{1/2}\Vert \nabla u\Vert^{1/2}, \\
\Vert u\Vert_{L^{4}({{{{\mathbb{R}}}}}^{3})}& \leq \left(\frac{4}{3\sqrt{3}}\right)^{3/4}\Vert
u\Vert ^{1/4}\Vert \nabla u\Vert ^{3/4}, \\
\Vert u\Vert _{L^{6}({{{{\mathbb{R}}}}}^{3})}& \leq \frac{2}{\sqrt{3}}\Vert
\nabla u\Vert.
\end{align*}
\end{lemma}
The standard explicitly skew-symmetrized trilinear form is defined by
\begin{equation}
\begin{aligned}
b^*(u,v,w)
&=\frac12 (u\cdot\nabla v,w)-\frac12(u\cdot\nabla w,v)  \\
&=(u\cdot\nabla v,w)+\frac12((\nabla\cdot u)v,w).
\end{aligned}
\end{equation}
Moreover, we define
\begin{equation}
M:=\sup_{u,v,w\in X}
\frac{b^*(u,v,w)}
{\|\nabla u\|\,\|\nabla v\|\,\|\nabla w\|}
<\infty.
\end{equation}
\begin{lemma}
[Lemma 6.14, p. 311 - p. 312, \cite{john2016finite}] \label{nonlinear_bound_John} There is a $C_2<\infty$ with
\begin{equation}
\begin{split}
 b^*(u,v, w)
&\leq C_2\sqrt{\|u\|}\sqrt{\|\nabla u\|} \|\nabla v\|\|\nabla w\|,
 \end{split}
\end{equation}
for $u,v,w \in (H^1_0(\Omega))^d$, $d=2$ or $3$.
\end{lemma}
\textbf{Assumption A1} on $I_H$: $I_{H}$ is an $L^{2}$ projection satisfying, for
all $\phi \in \left( H_{0}^{1}(\Omega )\right) ^{d},$
\begin{equation*}
\begin{split}
\|\phi -I_{H}\phi \|\leq C_{1}H\|\nabla \phi \|.
\end{split}
\end{equation*}

\begin{lemma}[H\"older's and Young's inequalities]\label{holderandyoung}
For any $\sigma>0$, $1\leq p\leq \infty$, $\frac{1}{p}+\frac{1}{q}=1$, the H\"older and Young's inequalities are as follows:
\begin{equation*}
    (u,v)\leq \|u\|_{L^p}\|v\|_{L^q},\ \ \text{and }
    (u,v)\leq \frac{\sigma}{p} \|u\|_{L^p}^p+
    \frac{\sigma^{-q/p}}{q} \|v\|_{L^q}^q.
\end{equation*}
\end{lemma}
Throughout the analysis, we use the inequalities in Lemma~\ref{holderandyoung} and the Poincar\'e inequality without further comment.

\begin{lemma}[Emmrich 1999 \cite{emmrich1999discrete}] \label{emmrich1999discrete}
Consider a sequence of inequalities 
\begin{equation}
    \frac{a_{n+1}-a_n}{\Delta t} \leq g_{n+1}+\lambda a_{n+1}, \quad n=0,1,\ldots,
\end{equation}
$\{a_n\}$, $\{g_n\} \subseteq \mathbb{R}$, and $a_0$ and $\{g_n\}$ are given, and $1-\lambda \Delta t>0$. Then, for $n=1,2,\ldots,$
\begin{equation}
a_n \leq (1-\lambda \Delta t)^{-n} \bigg[a_0 +\Delta t \sum_{j=0}^{n-1} (1-\lambda \Delta t)^{j} g_{j+1}\bigg].
\end{equation}
\end{lemma}

\section{The discrete problems}
We first recall the weak formulation of
\eqref{compressible_momentum}--\eqref{compressible_continuity}. Find
\((u,p) \in (X, Q)\) such that, for all $(w,\lambda)\in (X, Q)$,
\begin{align}
    (u_t,w) + b^*(u,u,w)+\nu (\nabla u,\nabla w)
    -(p,\nabla \cdot w)
    +\frac{1}{3}\nu(\nabla\cdot u,\nabla\cdot w)
    &= (f,w), \label{weak-compre-mom}\\
    \frac{1}{c^2}(p_t,\lambda)
    +\frac{1}{c^2}(u\cdot \nabla p,\lambda)
    +(\nabla \cdot u,\lambda)
    &=0. \label{weak-compre-conti}
\end{align}
Since $X^h\subset X$ and $Q^h\subset Q$, the exact solution also satisfies
\eqref{weak-compre-mom}--\eqref{weak-compre-conti} for all
$(w^h,\lambda^h)\in (X^h,Q^h)$.

The weak formulation of (\ref{momentum_eqn_v1})--(\ref{continuity_eqn_v1}) is as follows: Find $(v, q) \in (X, Q)$ such that, $\forall w \in X$, and $\forall \lambda \in Q$,
\begin{align}
     (v_t, w) +b^*(v,v,w)+\nu (\nabla v, \nabla w)-(q, \nabla \cdot w)-\chi (I_H(u-v), w) = (f,w),\\
     (\nabla \cdot v, \lambda)-\mu_1 (I_H(p-q), \lambda) -\mu_2 (I_H(q)-q, \lambda) =0.
\end{align}
The semi-discrete approximation of (\ref{momentum_eqn_v1})--(\ref{continuity_eqn_v1}) is ad follows. Denote $v^h(x,0)$ an approximation of $v(x,0)$. The approximated velocity and pressure are maps
\begin{equation}
    v^h:(0,T] \to X^h,\ q^h:(0,T]\to Q^h
\end{equation}
satisfying: $\forall w^h \in X^h$ and $\forall \lambda^h \in Q^h$,
\begin{equation}
\begin{split}
   \left(\frac{\partial v^h}{\partial t}, w^h\right)
+b^*(v^h,v^h,w^h)
+\nu (\nabla v^h, \nabla w^h)
-(q^h,\nabla\cdot w^h)\\
-\chi (I_H(u-v^h),w^h)
= (f,w^h), \label{nudged_time_conti_mom}
\end{split}
\end{equation}
\begin{align}
(\nabla\cdot v^h,\lambda^h)
-\mu_1(I_H(p-q^h),\lambda^h)
-\mu_2(I_Hq^h-q^h,\lambda^h)
=0.\label{nudged_time_conti_conti}
\end{align}
We present two fully discrete schemes, using linearized backward Euler and BDF2-IMEX time discretizations, respectively. They are described as follows:

linearized backward Euler: for all
$w^h\in X^h$, 
\begin{equation}
\begin{split}
\left(\frac{v_h^{n+1}-v_h^n}{\Delta t}, w^h\right)
+ \nu(\nabla v_h^{n+1}, \nabla w^h)
+ b^*(v_h^n, v_h^{n+1}, w^h) \\
- \chi\left(I_H(u(t_{n+1})) - I_H(v_h^{n+1}), w^h\right)
- (q_h^{n+1}, \nabla \cdot w^h)
= (f_h^{n+1}, w^h).\label{BE_momentum}
\end{split}
\end{equation}

BDF2-IMEX: for all
$w^h\in X^h$, 
\begin{equation}
\begin{split}
(\frac{3 v^{n+1}_h - 4v^{n}_h+v^{n-1}_h}{2\Delta t}, w^h) + \nu (\nabla v^{n+1}_h, \nabla w^h) + b^*( 2v^{n}_h - v^{n-1}_h, v^{n+1}_h, w^h)\\
- \chi (I_H(u(t_{n+1}))-I_H (v^{n+1}_h ), w^h)
-(q^{n+1}_h, \nabla \cdot w^h)
 =(f^{n+1}_h, w^h). \label{fully-momen}
\end{split}
\end{equation}
For both time discretizations, the pressure equation is given by: for all $\lambda^h\in Q^h$, 
\begin{equation}
(\nabla \cdot v^{n+1}_h, \lambda^h) - \mu_1 (I_H(p(t_{n+1})-q^{n+1}_h), \lambda^h) -\mu_2 (I_H(q^{n+1}_h)-q^{n+1}_h, \lambda^h) =0.\label{fully-conti}
\end{equation}
\subsection{Stability}
%here stability for semi-discrete and fully discrete case.
In this section, we first prove the stability of time-continuous case (\ref{nudged_time_conti_mom})--(\ref{nudged_time_conti_conti}). We also analyze the stability of the fully discrete linearized backward Euler, described in (\ref{BE_momentum}) and (\ref{fully-conti}), and BDF2-IMEX, described in (\ref{fully-momen})--(\ref{fully-conti}) of the nudged solutions. 

\begin{theorem} [Stability of $v^h$ and $q^h$] Assume $I_H$ is an $L^2$-projection. We have
\begin{equation}
\begin{split}
    \frac{d\|v^h\|^2}{dt}+\nu \|\nabla v^h\|^2 
    +\chi\|I_H(v^h)\|^2 +\alpha_1 \|q^h\|^2 \\
    \leq \chi \|I_H(u)\|^2 + \mu_1 \|I_H(p)\|^2 +\frac{1}{\nu} \|f\|^2_{-1}, \label{stability}
    \end{split}
\end{equation}
where $\alpha_1 = \min\{\mu_1, 2\mu_2\}$.
\end{theorem}
\begin{proof}
Take $w^h = v^h$, $\lambda^h=q^h$ in (\ref{nudged_time_conti_mom})--(\ref{nudged_time_conti_conti}), and add (\ref{nudged_time_conti_mom})
and (\ref{nudged_time_conti_conti}). We have
\begin{equation}
\begin{split}
\frac{1}{2}\frac{d\|v^h\|^2}{dt} +\nu \|\nabla v^h\|^2 +\chi (I_H(v^h), v^h) +\mu_1 (I_H(q^h), q^h) +\mu_2 (q^h-I_H(q^h), q^h) \\
= (f, v^h) + \chi (I_H(u), v^h)+\mu_1 (I_H(p), q^h).
    \end{split}
\end{equation}
Since $I_H$ is $L^2$-projection, $(I_H(q^h)-q^h, I_H(q^h))=0$. Thus
\begin{equation}
\begin{split}
    \frac{1}{2} \frac{d\|v^h\|^2}{dt} + \nu \|\nabla v^h\|^2 +\chi \|I_H(v^h)\|^2 +\mu_1 \|I_H(q^h)\|^2 +\mu_2\|q^h-I_H(q^h)\|^2 \\
    = (f, v^h) + \chi (I_H(u), I_H(v^h))+\mu_1 (I_H(p), I_H(q^h)).
    \end{split}
\end{equation}
For the term $\chi (I_H(u), I_H(v^h))$, and $\mu_1 (I_H(p), I_H(q^h))$, we use the polarization identity.
\begin{equation}
\begin{gathered}
     \chi (I_H(u), I_H(v^h))=\frac{\chi}{2}\|I_H(u)\|^2 + \frac{\chi}{2} \|I_H(v^h)\|^2 -\frac{\chi}{2}\|I_H(u-v^h)\|^2,\\
     \mu_1 (I_H(p), I_H(q^h))=\frac{\mu_1}{2}\|I_H(p)\|^2 + \frac{\mu_1}{2} \|I_H(q^h)\|^2 -\frac{\mu_1}{2}\|I_H(p-q^h)\|^2,\\
    (f,v^h)\leq \frac{\nu}{2}\|\nabla v^h\|^2 +\frac{1}{2\nu}\|f\|^2_{-1}.
    \end{gathered}
\end{equation}
Thus, we have
\begin{equation}
\begin{split}
    \frac{d\|v^h\|^2}{dt}+\nu \|\nabla v^h\|^2 
    +\chi\|I_H(v^h)\|^2 +\mu_1 \|I_H(q^h)\|^2  \\
    +2 \mu_2\|q^h-I_H(q^h)\|^2 \leq \chi \|I_H(u)\|^2 + \mu_1 \|I_H(p)\|^2 +\frac{1}{\nu} \|f\|^2_{-1}.
    \end{split}
\end{equation}
Using the identity $\|q^h\|^2=\|I_H(q^h)-q^h\|^2+ \|I_H(q^h)\|^2$, the stability of the discrete velocity $v^h$ and pressure $q^h$ follows (\ref{stability}). 
\end{proof}

\begin{theorem}[Linearized backward Euler stability]\label{BE_stability} Assume $I_H$ is an $L^2$-projection, we have
\begin{equation}
\begin{gathered}
   \|v^{N}_h\|^2-\|v^0_h\|^2 +\sum_{n=0}^{N-1} \|v^{n+1}_h -v^{n}_h\|^2
    + \Delta t \sum_{n=0}^{N-1} \big[\nu \|\nabla v^{n+1}_h\|^2+\chi\|I_H(v^{n+1}_h)\|^2 + \\
    \alpha_1\|q^{n+1}_h\|^2\big]
    \leq \Delta t \sum_{n=0}^{N-1}\big[\frac{1}{\nu}\|f(t_{n+1})\|^2_{-1} +\mu_1\|I_H(p(t_{n+1}))\|^2 + \chi \|I_H(u(t_{n+1}))\|^2 \big],
    \end{gathered}
\end{equation}
where $\alpha_1 = \min\{\mu_1, 2\mu_2\}$.
\end{theorem}
\begin{proof}
Take $w^h=v^{n+1}_h$, $\lambda^h=q^{n+1}_h$ in (\ref{BE_momentum}) and (\ref{fully-conti}), and add (\ref{BE_momentum}) and (\ref{fully-conti}), we have
\begin{equation}
\begin{gathered}
    \frac{1}{\Delta t}\left(v^{n+1}_h-v^{n}_h, v^{n+1}_h \right)+ \nu \|\nabla v^{n+1}_h\|^2 +\chi (I_H(v^{n+1}_h), v^{n+1}_h) \\
    +\mu_1 (I_H(q^{n+1}_h), q^{n+1}_h) + \mu_2 (q^{n+1}_h- I_H(q^{n+1}_h), q^{n+1}_h)\\
    = (f(t_{n+1}), v^{n+1}_h)+\mu_1 (I_H(p(t_{n+1})), q^{n+1}_h) +\chi (I_H(u(t_{n+1})), v^{n+1}_h).\label{BE-IMEX_start}
    \end{gathered}
\end{equation}
First, we apply the polarization identity to $\frac{1}{\Delta t}\left(v^{n+1}_h-v^{n}_h, v^{n+1}_h \right)$, then
\begin{equation}
    \frac{1}{\Delta t}\left(v^{n+1}_h-v^{n}_h, v^{n+1}_h \right)=\frac{1}{2\Delta t}\left(\|v^{n+1}_h\|^2 - \|v^{n}_h\|^2 + \|v^{n+1}_h-v^{n}_h\|^2 \right).
\end{equation}
Since $I_H$ is an $L^2$-projection, 
\begin{equation}
\begin{gathered}
    \mu_2 (q^{n+1}_h - I_H (q^{n+1}_h), q^{n+1}_h)
    =\mu_2 \|I_H(q^{n+1}_h)-q^{n+1}_h\|^2.
    \end{gathered}
\end{equation}
\begin{equation}
\begin{gathered}
    (f(t_{n+1}), v^{n+1}_h)\leq \frac{1}{2\nu}\|f(t_{n+1})\|^2_{-1}+ \frac{\nu}{2}\|\nabla v^{n+1}_h\|^2,\\
     \chi (I_H(u(t_{n+1})), v^{n+1}_h) 
     \leq \frac{\chi}{2} \|I_H(u(t_{n+1}))\|^2 + \frac{\chi}{2} \|I_H(v^{n+1}_h)\|^2,\\
    \mu_1 (I_H(p(t_{n+1})), q^{n+1}_h)
    \leq \frac{\mu_1}{2}\|I_H(p(t_{n+1}))\|^2 + \frac{\mu_1}{2}\|I_H(q^{n+1}_h))\|^2. \label{pq_inequalty}
    \end{gathered}
\end{equation}
Thus, we have
\begin{equation}
\begin{split}
\frac{1}{2\Delta t}\left(\|v^{n+1}_h\|^2 - \|v^{n}_h\|^2 + \|v^{n+1}_h-v^{n}_h\|^2 \right)
     +\frac{\nu}{2} \|\nabla v^{n+1}_h\|^2 + 
     \frac{\chi}{2} \|I_H(v^{n+1}_h)\|^2
     \\
     +\frac{1}{2} \alpha_1 \|q^{n+1}_h\|^2
     \leq \frac{1}{2\nu} \|f(t_{n+1})\|^2_{-1} + \frac{\mu_1}{2} \|I_H(p(t_{n+1}))\|^2 +\frac{\chi}{2}\|I_H(u(t_{n+1}))\|^2.\label{stability-before-sum-be}
     \end{split}
\end{equation}
Summing over (\ref{stability-before-sum-be}) from $n=1$ to $n=N-1$, and multiplying it by $2\Delta t$, we have the stability.
\end{proof}
\begin{theorem}[BDF2-IMEX stability] Assume $I_H$ is an $L^2$-projection.
\begin{equation}
\begin{gathered}
   \|v^{N}_h\|^2 + \|2v^N_h-v^{N-1}_h\|^2     -\|v^0_h\|^2 - \|2 v^1_h-v^0_h\|^2 +\sum_{n=0}^{N-1} \|v^{n+1}_h -2v^{n}_h+v^{n-1}_h\|^2\\
    + 2\Delta t \sum_{n=0}^{N-1} \bigg[\nu \|\nabla v^{n+1}_h\|^2+\chi\|I_H(v^{n+1}_h)\|^2 + \alpha_1 \|q^{n+1}_h\|^2\bigg]\\
    \leq 2\Delta t \sum_{n=0}^{N-1}\bigg[\frac{1}{\nu}\|f(t_{n+1})\|^2_{-1} +\mu_1\|I_H(p(t_{n+1}))\|^2 + \chi \|I_H(u(t_{n+1}))\|^2 \bigg],
    \end{gathered}
\end{equation}
where $\alpha_1 = $ $\min\{\mu_1, 2\mu_2\}$.
\end{theorem}
\begin{proof}
Take $w^h=v^{n+1}_h$, $\lambda^h=q^{n+1}_h$ in (\ref{fully-momen})--(\ref{fully-conti}), and add (\ref{fully-momen})--(\ref{fully-conti}), we have
\begin{equation}
\begin{split}
    \left( \frac{3 v^{n+1}_h - 4v^n_h +v^{n-1}_h}{2 \Delta t }, v^{n+1}_h \right)+ \nu \|\nabla v^{n+1}_h\|^2 +\chi (I_H(v^{n+1}_h), v^{n+1}_h) \\
    +\mu_1 (I_H(q^{n+1}_h), q^{n+1}_h) + \mu_2 (q^{n+1}_h- I_H(q^{n+1}_h), q^{n+1}_h)\\
    = (f(t_{n+1}), v^{n+1}_h)+\mu_1 (I_H(p(t_{n+1})), q^{n+1}_h) +\chi (I_H(u(t_{n+1})), v^{n+1}_h).\label{BDF2-IMEX_start}
    \end{split}
\end{equation}
First, we bound $\frac{1}{2\Delta t} \left(3 v^{n+1}_h - 4v^n_h+v^{n-1}_h, v^{n+1}_h \right)$:
\begin{equation}
    3 v^{n+1}_h - 4 v^{n}_h+ v^{n-1}_h= (v^{n+1}_h - v^{n}_h) + (2v^{n+1}_h - v^{n}_h) - (2v^{n}_h-v^{n-1}_h).
\end{equation}
Using the polarization identity to each of the three groups gives,  and multiplying by $\frac{1}{2\Delta t}$, we have
\begin{equation}
\begin{gathered}
     (\frac{3 v^{n+1}_h - 4v^n_h +v^{n-1}_h}{2 \Delta t }, v^{n+1}_h ) = \frac{1}{4 \Delta t } \big[\|v^{n+1}_h\|^2 -\|v^{n}_h\|^2 \\
     + \|2v^{n+1}_h-v^n_h\|^2 - \|2v^{n}_h-v^{n-1}_h\|^2 +\|v^{n+1}_h - 2v^{n}_h + v^{n-1}_h\|^2 \big].\label{time-inequality}
     \end{gathered}
\end{equation}
We estimate 
\begin{equation}
\begin{gathered}
      \mu_2 (q_h^{n+1}-I_H(q_h^{n+1}),q_h^{n+1}), 
\chi(I_H(v_h^{n+1}),v_h^{n+1}),\\
\chi(I_H(u(t_{n+1})),v_h^{n+1}), (f^{n+1},v_h^{n+1}),
\mu_1(I_H(p(t_{n+1})),q_h^{n+1}),  
\end{gathered}
\end{equation}
in the same way as in Theorem~\ref{BE_stability}. Thus (\ref{BDF2-IMEX_start}) becomes
\begin{equation}
\begin{split}
\frac{1}{4 \Delta t } \big[\|v^{n+1}_h\|^2 -\|v^{n}_h\|^2 
     + \|2v^{n+1}_h-v^n_h\|^2 - \|2v^{n}_h-v^{n-1}_h\|^2 \\
     +\|v^{n+1}_h - 2v^{n}_h + v^{n-1}_h\|^2 \big]
     +\frac{\nu}{2} \|\nabla v^{n+1}_h\|^2 + 
     \frac{\chi}{2} \|I_H(v^{n+1}_h)\|^2
     + \alpha_1 \|q^{n+1}_h\|^2 \\
     \leq \frac{1}{2\nu} \|f(t_{n+1})\|^2_{-1} + \frac{\mu_1}{2}\|I_H(p(t_{n+1}))\|^2 +\frac{\chi}{2}\|I_H(u(t_{n+1}))\|^2.\label{stability-before-sum}
     \end{split}
\end{equation}
Summing over (\ref{stability-before-sum}) from $n=1$ to $n=N-1$, and multiplying it by $4\Delta t$, we have the stability.
\end{proof}

\section{Error estimates}\label{error-estimates}

The continuous error estimate was analyzed in \cite{CibikFang26} under suitable regularity and parameter assumptions. In this work, we focus on the numerical analysis of finite element approximations of the proposed model. Section \ref{semi-discrete-error} presents the semi-discrete error analysis, where the problem is continuous in time and discretized only in space by finite elements. Section \ref{full-discrete-error} establishes the fully discrete error estimate for the linearized backward Euler scheme. Throughout this section, the explicit Sobolev constants appearing in the analysis are those corresponding to the three-dimensional case. The regularity assumptions \eqref{smoothAssum-semidiscrete} and \eqref{smoothAssum-fullydiscrete} are therefore stated for $d=3$; the corresponding two-dimensional statements are obtained by using the sharper constants of Lemma~\ref{norms_property}, and all estimates remain valid with the same rates.

Our analysis establishes an infinite predictability horizon: the effect of the initial error decays exponentially in time at a rate of $\mathcal{O}(\min\{\chi,\mu_1\})$, while the model error is bounded by $\mathcal{O}(H,\mu_1^{-1/2})$. Thus, balancing these two contributions suggests choosing $\mu_1=\mathcal{O}(H^{-2})$ to obtain the optimal convergence rate.

\subsection{Semi-discrete error estimates}\label{semi-discrete-error} 
Define the velocity error $e^h=u-v^h=\eta-\phi^h$, where $\eta = u-\Tilde{v}$, $\phi^h= v^h-\Tilde{v}$, $\forall \Tilde{v}\in X^h$, and the pressure error $e^h_p=p-q^h =\eta_p-\phi^h_p$, where $\eta_p = p-\Tilde{q}$ and $\phi^h_p= q^h-\Tilde{q}$, $\forall \Tilde{q}\in Q^h$.
\begin{theorem} We assume 
\begin{equation}
\begin{split}
    \chi-\frac{27}{16}\frac{C_2^4}{\nu^3}\|\nabla v^h\|^4_{L^\infty(0,T; L^2(\Omega))}\geq \alpha \chi,\quad
    \frac{\nu}{16}- \chi C_1^2H^2\geq 0, \quad \text{and} \quad \mu_1\geq \mu_2,\\
    \end{split}
\label{assumptions}
\end{equation}
and the following regularity assumptions:
\begin{equation}\label{smoothAssum-semidiscrete}
\begin{gathered}
u \in L^\infty(0,T;(H^{m+1}(\Omega))^3)
    \cap L^4(0,T; (H^{m+1}(\Omega))^3),\
u_t \in L^2(0,T; (H^{m+1}(\Omega))^3),\\
p \in L^\infty(0,T;H^{m}(\Omega))
    \cap L^4(0,T;W^{1,3}(\Omega)),
p_t \in L^2(0,T;H^{m}(\Omega)),\\
q^h_t \in L^2(0,T;L^2(\Omega)).
\end{gathered}
\end{equation}
Then, we have the following error estimate:
\allowdisplaybreaks
\begin{align}
    &\quad \frac{1}{2}\|e^h\|^2 +\frac{1}{2c^2}\|e^h_p\|^2
    +\frac{\nu}{16}\int_{0}^t \exp [-\beta (t-t')]\|\nabla e^h\|^2 \, dt'
    \nonumber\\
    &\leq
    2\exp [-\beta t]
    \big[\|e^h(0)\|^2+\frac{1}{c^2}\|e^h_p(0)\|^2 \big]
    +\frac{(\mu_1+4)}{\beta} C_1^2 H^2
    \max_{0\leq t \leq T}\|\nabla p\|^2
    \nonumber\\
    &\quad
    +\big[8\frac{\mu_1}{\beta}+3\big]
    \max_{0\leq t\leq T}\|\eta_p\|^2
    +\big[\frac{\chi}{\beta}+3\big]
    \max_{0\leq t\leq T}\|\eta\|^2
    +4C_1^2 H^2
    \|\nabla \eta_p\|^2_{L^2(0,T;L^2(\Omega))}
    \nonumber\\
    &\quad
    +\frac{4}{\nu}M^2
    \Big[
    \|\nabla u\|^2_{L^4(0,T;L^2(\Omega))}
    +\|\nabla v^h\|^2_{L^4(0,T;L^2(\Omega))}
    \Big]
    \|\nabla \eta\|^2_{L^4(0,T;L^2(\Omega))}
    \nonumber\\
    &\quad
    +\frac{6}{\nu}
    \|\eta_p\|^2_{L^2(0,T;L^2(\Omega))}
    +\frac{16}{\mu_1 c^4}
    \|\eta_{p,t}\|^2_{L^2(0,T;L^2(\Omega))}
    \nonumber\\
    &\quad
    +\Big[
    \frac{16}{\mu_1}
    +\frac{8}{\nu}
    +\frac{7\nu}{12}
    \Big]
    \|\nabla \eta\|^2_{L^2(0,T;L^2(\Omega))}
    +\frac{8}{\nu}
    \int_{0}^T \|\eta_t\|^2_{-1}\,dt
    \nonumber\\
    &\quad
    +\frac{16}{\mu_1 c^4}
    \Big[
    \|q^h_t\|^2_{L^2(0,T;L^2(\Omega))}
    +\frac{2}{3}
    \|\nabla u\|^4_{L^4(0,T;L^2(\Omega))}
    +\frac{2}{3}
    \|\nabla p\|^4_{L^4(0,T;L^3(\Omega))}
    \Big],
    \label{error-estimate-result}
\end{align}
where $\beta=\min\{\alpha \chi, \frac{c^2\mu_1}{16}\}$.
\begin{corollary}
Under the regularity assumptions~(\ref{smoothAssum-semidiscrete}), the parameter assumptions~(\ref{assumptions}), and the approximation properties~(\ref{prop}), the following estimate holds:
\allowdisplaybreaks
\begin{align}
&\frac{1}{2}\|e^h\|^2
+\frac{1}{2c^2}\|e^h_p\|^2
+\frac{\nu}{16}\int_{0}^t
\exp[-\beta (t-t')]
\|\nabla e^h\|^2\,dt'
\nonumber\\
&\leq
2\exp[-\beta t]
\left[
\|e^h(0)\|^2
+\frac{1}{c^2}\|e^h_p(0)\|^2
\right]
+\frac{(\mu_1+4)}{\beta}C_1^2H^2
\max_{0\le t\le T}\|\nabla p\|^2
\nonumber\\
&\quad
+\left[
8\frac{\mu_1}{\beta}+3
\right]
Ch^{2m}
\max_{0\le t\le T}|p|_{H^m}^2
+\left[
\frac{\chi}{\beta}+3
\right]
Ch^{2m+2}
\max_{0\le t\le T}|u|_{H^{m+1}}^2
\nonumber\\
&\quad
+4C_1^2H^2
Ch^{2m-2}
|p|^2_{L^2(0,T;H^m(\Omega))}
+\left[
\frac{16}{\mu_1}
+\frac{8}{\nu}
+\frac{7\nu}{12}
\right]
Ch^{2m}
|u|^2_{L^2(0,T;H^{m+1}(\Omega))}
\nonumber\\
&\quad
+\frac{4}{\nu}M^2
\left[
\|\nabla u\|^2_{L^4(0,T;L^2(\Omega))}
+\|\nabla v^h\|^2_{L^4(0,T;L^2(\Omega))}
\right]
Ch^{2m}
|u|^2_{L^4(0,T;H^{m+1}(\Omega))}
\nonumber\\
&\quad
+\frac{6}{\nu}
Ch^{2m}
|p|^2_{L^2(0,T;H^m(\Omega))}
+\frac{16}{\mu_1c^4}
Ch^{2m}
|p_t|^2_{L^2(0,T;H^m(\Omega))}
+\frac{8}{\nu}
Ch^{2m+2}
|u_t|^2_{L^2(0,T;H^{m+1}(\Omega))}
\nonumber\\
&\quad
+\frac{16}{\mu_1c^4}
\left[
\|q_t^h\|^2_{L^2(0,T;L^2(\Omega))}
+\frac{2}{3}\|\nabla u\|^4_{L^4(0,T;L^2(\Omega))}
+\frac{2}{3}\|\nabla p\|^4_{L^4(0,T;L^3(\Omega))}
\right]
\nonumber,
\label{error-estimate-result-after-prop}
\end{align}
where $\beta=\min\{\alpha \chi, \frac{c^2\mu_1}{16}\}$.
\end{corollary}

\begin{corollary}
Under the regularity assumptions~(\ref{smoothAssum-semidiscrete}), the parameter assumptions~(\ref{assumptions}), the approximation properties~(\ref{prop}), and the choice $\mu_1=\mathcal{O}(H^{-2})$, the following estimate holds:
\begin{equation}
\begin{gathered}
   \frac{1}{2} \|e^h\|^2 +\frac{1}{2c^2}\|e^h_p\|^2 +\frac{\nu}{16}\int_{0}^t \exp[-\beta (t-t')]\|\nabla e^h\|^2 \, dt' \\
\leq 2\exp[-\beta t] \big[\|e^h(0)\|^2+\frac{1}{c^2}\|e^h_p(0)\|^2 \big]+
C\big[ H^2+h^{2m}+ H^2h^{2m-2}+\frac{H^2}{c^4}\big],
\label{error-estimate-result-after-prop-chi-mu_1}
\end{gathered}
\end{equation}
where $\beta=\min\{\alpha \chi, \frac{c^2\mu_1}{16}\}$.
\end{corollary}
\end{theorem}

\begin{remark}\label{remark-semi}
If $\frac{c^2\mu_1}{16}\leq \alpha\chi$, then
$\beta=\min\{\alpha \chi, \frac{c^2\mu_1}{16}\}=\frac{c^2\mu_1}{16}$, and hence
\begin{equation}
\begin{gathered}
   \frac{\mu_1+4}{\beta}C_1^2H^2
\max_{0\le t\le T}\|\nabla p\|^2
=
\frac{16(\mu_1+4)}{c^2\mu_1}C_1^2H^2
\max_{0\le t\le T}\|\nabla p\|^2 \\
=
\left(16+\frac{64}{\mu_1}\right)
\frac{C_1^2H^2}{c^2}
\max_{0\le t\le T}\|\nabla p\|^2. 
\end{gathered}
\end{equation}
Therefore, this error
contribution scales as $\mathcal{O}(H^2/c^2)$ rather than $\mathcal{O}(H^2)$.
\end{remark}

\begin{proof}
To derive the error equation, we subtract the weak form of the nudged system \eqref{nudged_time_conti_mom}--\eqref{nudged_time_conti_conti} from \eqref{weak-compre-mom}--\eqref{weak-compre-conti}. We choose the test functions $w^h=\phi^h$, and $\lambda^h=\phi^h_p$, adding the resulting equations gives:
\begin{equation}
\begin{gathered}
    \frac{1}{2} \frac{d\|\phi^h\|^2}{dt} +\nu \|\nabla \phi^h\|^2 +\frac{\nu}{3}\|\nabla \cdot \phi^h\|^2 +\chi (I_H(\phi^h), \phi^h)\\
     -\mu_1 (I_H(p-q^h), \phi^h_p)-\mu_2 (I_H(q^h)-q^h, \phi^h_p)\\
    =(\eta_t, \phi^h) +\nu (\nabla \eta,\nabla \phi^h) +\frac{\nu}{3}(\nabla \cdot \eta, \nabla \cdot \phi^h)
    + \chi (I_H(\eta), \phi^h)  
 +(\nabla \cdot \eta, \phi^h_p)\\  -(\eta_p, \nabla \cdot \phi^h) + b^*(u,u,\phi^h)-b^*(v^h,v^h, \phi^h) +\frac{1}{c^2}( \frac{\partial p}{\partial t} + u\cdot \nabla p, \phi^h_p ).\label{error-equation}
 \end{gathered}
\end{equation}
The key contribution is handling the pressure--related terms and the advection term from the continuity equation of the slightly compressible flow.
\begin{equation}
\begin{gathered}
    -\mu_1 (I_H(p-q^h), \phi^h_p)-\mu_2 (I_H(q^h)-q^h,\phi^h_p)\\
    = -\mu_1 (I_H(p)-q^h, \phi^h_p)-(\mu_1-\mu_2) ((I-I_H)(q^h), \phi^h_p).
    \end{gathered}
\end{equation}
We bound the above terms as follows:
\begin{equation}
\begin{gathered}
    -\mu_1 (I_H(p)-q^h, \phi^h_p)=-\mu_1(I_H(p)-p,\phi^h_p) -\mu_1(p-q^h,\phi^h_p)\\
    \geq -\frac{\mu_1}{2}C_1^2 H^2 \|\nabla p\|^2 -\frac{\mu_1}{2}\|\phi^h_p\|^2 + \mu_1 \|\phi^h_p\|^2 -\mu_1 (\eta_p, \phi^h_p)\\
    \geq -\frac{\mu_1}{2}C_1^2 H^2 \|\nabla p\|^2 +  \frac{\mu_1}{4} \|\phi^h_p\|^2 - 4\mu_1 \|\eta_p\|^2,
    \end{gathered}
\end{equation}
\begin{equation}
\begin{gathered}
    -(\mu_1-\mu_2) ((I-I_H)(q^h), \phi^h_p)\\
    =-(\mu_1-\mu_2) ((I-I_H)(q^h-p),\phi^h_p)
    -(\mu_1-\mu_2) ((I-I_H)(p),\phi^h_p)\\
    =(\mu_1 -\mu_2)\big[\|\phi^h_p\|^2- (I_H(\phi^h_p),\phi^h_p)\big] -(\mu_1-\mu_2) ((I-I_H)(\eta_p),\phi^h_p) \\
    -(\mu_1-\mu_2) ((I-I_H)(p),\phi^h_p).
    \end{gathered}    
\end{equation}
Assume $\mu_1\geq \mu_2$. Since $I_H$ is an $L^2$-projection,
\begin{equation}
    (\mu_1-\mu_2)\left(\|\phi^h_p\|^2- (I_H(\phi^h_p),\phi^h_p)\right)=(\mu_1-\mu_2)(\|\phi^h_p\|^2- \|I_H(\phi^h_p)\|^2)\geq 0.
\end{equation}
Thus, we have
\begin{equation}
\begin{gathered}
    -(\mu_1-\mu_2) ((I-I_H)(q^h), \phi^h_p)\geq -(\mu_1-\mu_2) ((I-I_H)(\eta_p),\phi^h_p) \\
    -(\mu_1-\mu_2) ((I-I_H)(p),\phi^h_p)\\
    \geq -\frac{2 (\mu_1-\mu_2)}{\mu_1} C_1^2 H^2 (\|\nabla \eta_p\|^2+\|\nabla p\|^2) -\frac{\mu_1}{8}\|\phi^h_p\|^2.
    \end{gathered}
    \end{equation}
Hence,
\begin{equation}
\begin{gathered}
    -\mu_1 (I_H(p-q^h), \phi^h_p)-\mu_2 (I_H(q^h)-q^h,\phi^h_p) \geq 
    -\frac{\mu_1}{2}C_1^2 H^2 \|\nabla p\|^2 \\+  \frac{\mu_1}{8} \|\phi^h_p\|^2 - 4\mu_1 \|\eta_p\|^2
-\frac{2 (\mu_1-\mu_2)}{\mu_1} C_1^2 H^2 (\|\nabla \eta_p\|^2+\|\nabla p\|^2).\label{estimate-pressure}
    \end{gathered}
    \end{equation}
For the following terms, we have
\begin{equation}
\begin{gathered}
     (\nabla \cdot \eta, \phi^h_p)\leq \frac{\mu_1}{32}\|\phi^h_p\|^2 + \frac{8}{\mu_1} \|\nabla \eta\|^2,\\
    -(\eta_p, \nabla \cdot \phi^h)\leq \frac{3}{\nu}\|\eta_p\|^2 +\frac{\nu}{12}\|\nabla \cdot \phi^h\|^2,
    \end{gathered}
\end{equation}
\begin{equation}
\begin{split}
    \frac{1}{c^2} \left( \frac{\partial p}{\partial t}, \phi^h_p \right) = \frac{1}{c^2} \left( \frac{\partial (p-q^h)}{\partial t}, \phi^h_p \right)+\frac{1}{c^2} \left( \frac{\partial q^h}{\partial t}, \phi^h_p \right)\\
    = -\frac{1}{2c^2}  \frac{d \|\phi^h_p\|^2}{d t}+ \frac{1}{c^2} (\eta_{p,t}, \phi^h_p) +\frac{1}{c^2} \left( \frac{\partial q^h}{\partial t}, \phi^h_p \right)\\
    \leq -\frac{1}{2c^2}  \frac{d \|\phi^h_p\|^2}{d t} +\frac{8}{\mu_1c^4}(\|\eta_{p,t}\|^2+\|q^h_t\|^2) + \frac{\mu_1}{32}\|\phi^h_p\|^2,
    \end{split}
\end{equation}
\begin{equation}
\begin{split}
\frac{1}{c^2}\left(u\cdot \nabla p, \phi^h_p \right) \leq \frac{1}{c^2} \|u\|_{L^6} \|\nabla p\|_{L^3}\|\phi^h_p\| 
    \leq \frac{2}{\sqrt{3}}\frac{1}{c^2}\|\nabla u\| \|\nabla p\|_{L^3} \|\phi^h_p\| \\
    \leq \frac{1}{c^4}\frac{32}{3}\frac{1}{\mu_1} \|\nabla u\|^2 \|\nabla p\|^2_{L^3} + \frac{\mu_1}{32}\|\phi^h_p\|^2
    \leq \frac{16}{3c^4\mu_1} \left(\|\nabla u\|^4 + \|\nabla p\|^4_{L^3}\right) + \frac{\mu_1}{32}\|\phi^h_p\|^2.\label{estimate-pressure-2}
    \end{split}
\end{equation}
Now we handle the nonlinear terms $ b^*(u,u,\phi^h)-b^*(v^h,v^h, \phi^h)$ as follows:
\begin{equation}
\begin{gathered}
b^*(u,u,\phi^h)-b^*(v^h,v^h, \phi^h) = b^*(u,u, \phi^h) + b^*(\phi^h-v^h,v^h, \phi^h) - b^*(\phi^h,v^h, \phi^h)\\
=b^*(u,u-v^h,\phi^h) + b^*(\eta, v^h,\phi^h) -b^*(\phi^h,v^h,\phi^h)\\
=b^*(u,\eta,\phi^h) + b^*(\eta, v^h,\phi^h) -b^*(\phi^h,v^h,\phi^h).
\end{gathered}
\end{equation}
We estimate the above terms as follows:
\begin{equation}
\begin{gathered}
    b^*(u,\eta,\phi^h)\leq M\|\nabla u\|\|\nabla \eta\|\|\nabla \phi^h\|\leq \frac{2}{\nu}M^2 \|\nabla u\|^2 \|\nabla \eta\|^2 +\frac{\nu}{8}\|\nabla \phi^h\|^2,\\
    b^*(\eta, v^h,\phi^h) \leq M\|\nabla \eta\|\|\nabla v^h\|\|\nabla \phi^h\|\leq \frac{2}{\nu}M^2 \|\nabla \eta\|^2 \|\nabla v^h\|^2 +\frac{\nu}{8}\|\nabla \phi^h\|^2,\\
    b^*(\phi^h,v^h,\phi^h)\leq C_2 \|\phi^h\|^{1/2}\|\nabla v^h\|\|\nabla \phi^h\|^{3/2}\leq \frac{\nu}{2}\|\nabla \phi^h\|^2 + \frac{27}{32}\frac{C_2^4}{\nu^3}\|\nabla v^h\|^4 \|\phi^h\|^2.
    \end{gathered}
\end{equation}
Thus, 
\begin{equation}
    \begin{gathered}
    b^*(u,u,\phi^h)-b^*(v^h,v^h, \phi^h)\leq \frac{3\nu}{4}\|\nabla \phi^h\|^2 \\
    +\frac{2}{\nu}M^2(\|\nabla u\|^2+\|\nabla v^h\|^2) \|\nabla \eta\|^2
    + \frac{27}{32}\frac{C_2^4}{\nu^3}\|\nabla v^h\|^4 \|\phi^h\|^2.\label{estimate-nonlinear}
    \end{gathered}
\end{equation}
Now, we take care of $ \nu (\nabla \eta,\nabla \phi^h)$, $(\eta_t, \phi^h)$, and $\frac{\nu}{3}(\nabla \cdot \eta, \nabla \cdot \phi^h)$.
\begin{equation}
\begin{gathered}
    \nu (\nabla \eta,\nabla \phi^h)\leq \frac{\nu}{16}\|\nabla \phi^h\|^2 + \frac{4}{\nu}\|\nabla \eta\|^2.\\
    (\eta_t, \phi^h)\leq \|\nabla \eta_t\|_{-1}\|\nabla \phi^h\|\leq \frac{4}{\nu}\|\eta_t\|^2_{-1} + \frac{\nu}{16}\|\nabla \phi^h\|^2,\\
    \frac{\nu}{3}(\nabla \cdot \eta, \nabla \cdot \phi^h)\leq \frac{\nu}{3}\|\nabla \cdot \eta\|\|\nabla \cdot \phi^h\|\leq \frac{\nu}{6}\|\nabla  \eta\|^2 +  \frac{\nu}{6}\|\nabla \cdot \phi^h\|^2.
    \end{gathered}
\end{equation}
Lastly, we handle the nudging terms for velocity.
\begin{equation}
\begin{split}
    \chi (I_H(\phi^h), \phi^h)=  \chi\|\phi^h\|^2 -\chi \|(I-I_H)(\phi^h)\|^2
    \geq \chi\|\phi^h\|^2 -\chi C_1^2H^2 \|\nabla \phi^h\|^2.
    \end{split}
\end{equation}
\begin{equation}
    \chi (I_H(\eta), \phi^h) \leq  \chi\|I_H(\eta)\| \|\phi^h\|\leq \frac{\chi}{2}\|I_H(\eta)\|^2 +\frac{\chi}{2}\|\phi^h\|^2.\label{estimate-velocity-nudging}
\end{equation}
Adding (\ref{estimate-pressure})--(\ref{estimate-pressure-2}), (\ref{estimate-nonlinear})--(\ref{estimate-velocity-nudging}), and multiplying it by $2$, the error equation (\ref{error-equation}) becomes
\begin{equation}
\begin{split}
    \frac{d(\|\phi^h\|^2+1/c^2 \|\phi^h_p\|^2)}{dt} +\frac{\nu}{8} \|\nabla \phi^h\|^2 +\frac{\nu}{6}\|\nabla \cdot \phi^h\|^2 +\frac{\mu_1}{16}\|\phi^h_p\|^2\\
    +\left(\chi-\frac{27}{16}\frac{C_2^4}{\nu^3}\|\nabla v^h\|^4\right)\|\phi^h\|^2 + \left(\frac{\nu}{8}- 2\chi C_1^2H^2 \right)\|\nabla \phi^h\|^2
\leq      \mu_1 C_1^2 H^2 \|\nabla p\|^2 \\
+ 8\mu_1 \|\eta_p\|^2
+\frac{4 (\mu_1-\mu_2)}{\mu_1} C_1^2 H^2 (\|\nabla \eta_p\|^2+\|\nabla p\|^2) 
+ \frac{16}{\mu_1} \|\nabla \eta\|^2+\frac{6}{\nu}\|\eta_p\|^2 \\
+\frac{16}{\mu_1c^4}\|\eta_{p,t}\|^2+\frac{16}{\mu_1c^4}\|q^h_t\|^2 + \frac{1}{c^4}\frac{1}{\mu_1}\frac{32}{3} \left(\|\nabla u\|^4 + \|\nabla p\|^4_{L^3}\right)+\frac{8}{\nu}\|\nabla \eta\|^2\\
+\frac{4}{\nu}M^2(\|\nabla u\|^2+\|\nabla v^h\|^2) \|\nabla \eta\|^2
 +\frac{8}{\nu}\|\eta_t\|^2_{-1}+\frac{\nu}{3}\|\nabla  \eta\|^2
+\chi\|I_H(\eta)\|^2.
 \end{split}
 \label{error-estimate-1}
\end{equation}
With $H$ and $\chi$ conditions (\ref{assumptions}), and rearranging the right hand side of (\ref{error-estimate-1}),
\begin{equation}
\begin{gathered}
    \frac{d(\|\phi^h\|^2+1/c^2\|\phi^h_p\|^2)}{dt} +\frac{\nu}{8} \|\nabla \phi^h\|^2 +\alpha \chi \|\phi^h\|^2 +\frac{\mu_1}{16}\|\phi^h_p\|^2 
\leq \mu_1 C_1^2 H^2 \|\nabla p\|^2\\ + 4 C_1^2 H^2 \|\nabla p\|^2  +8\mu_1 \|\eta_p\|^2 +\frac{6}{\nu} \|\eta_p\|^2
+4C_1^2 H^2 \|\nabla \eta_p\|^2 
+\frac{16}{\mu_1c^4}\|\eta_{p,t}\|^2\\
+\frac{16}{\mu_1c^4}\|q^h_t\|^2 
+ \big[\frac{16}{\mu_1}+\frac{8}{\nu}+\frac{\nu}{3}
+ \frac{4}{\nu}M^2(\|\nabla u\|^2+\|\nabla v^h\|^2) \big] \|\nabla \eta\|^2\\
+ \frac{1}{c^4}\frac{1}{\mu_1}\frac{32}{3} \big[\|\nabla u\|^4 + \|\nabla p\|^4_{L^3}\big]
 +\frac{8}{\nu}\|\eta_t\|^2_{-1}
+\chi\|\eta\|^2.
 \end{gathered}
 \label{error-estimate-2}
\end{equation}
Using the integration factor $\exp [\beta t]$, where $\beta=\min\{\alpha \chi, \frac{c^2\mu_1}{16}\}$. Integrating from $0$ to $t$, and multiplying $\exp[-\beta t ]$ to both sides of (\ref{error-estimate-2}) gives:
\allowdisplaybreaks
\begin{align}
&    \|\phi^h\|^2 +\frac{1}{c^2}\|\phi^h_p\|^2 +\frac{\nu}{8}\int_{0}^t \exp [-\beta (t-t')]\|\nabla \phi^h\|^2 \, dt'\nonumber\\
&\quad
\leq \exp[-\beta t] \left(\|\phi^h(0)\|^2 +\frac{1}{c^2}\|\phi^h_p(0)\|^2 \right) +\int_{0}^t \exp [-\beta (t-t')] \bigg[ 
\mu_1 C_1^2 H^2 \|\nabla p\|^2\nonumber\\
&\quad
+ 4 C_1^2 H^2 \|\nabla p\|^2 
+8\mu_1 \|\eta_p\|^2 +\frac{6}{\nu} \|\eta_p\|^2 +4C_1^2 H^2 \|\nabla \eta_p\|^2 
+\frac{16}{\mu_1c^4}\|\eta_{p,t}\|^2
\nonumber\\
&\quad
+\frac{16}{\mu_1c^4}\|q^h_t\|^2
+ \left(\frac{16}{\mu_1}+\frac{8}{\nu}+\frac{\nu}{3}
+ \frac{4}{\nu}M^2(\|\nabla u\|^2+\|\nabla v^h\|^2) \right) \|\nabla \eta\|^2\nonumber\\
&\quad
+ \frac{1}{c^4}\frac{1}{\mu_1}\frac{32}{3} \left(\|\nabla u\|^4 + \|\nabla p\|^4_{L^3}\right)
 +\frac{8}{\nu}\|\eta_t\|^2_{-1}
+\chi\|\eta\|^2
\bigg]\, dt'.\label{error-estimate-after-IBP}
\end{align}

We handle $\int_{0}^t \exp [-\beta (t-t')] (\mu_1+4) C_1^2 H^2 \|\nabla p\|^2\, dt'$ term as the following:
\begin{equation}
\begin{gathered}
    \int_{0}^t \exp [-\beta (t-t')] (\mu_1+4) C_1^2 H^2 \|\nabla p\|^2\, dt' 
\\
\leq \max_{0\leq t \leq T} C_1^2 H^2\|\nabla p\|^2 (\mu_1+4) \exp [-\beta t] \int_{0}^t \exp [\beta t']\, dt'\\
=\max_{0\leq t \leq T} C_1^2 H^2\|\nabla p\|^2 (\mu_1+4) \exp [-\beta t] \frac{\exp [\beta t]-1}{\beta}\\
=\max_{0\leq t \leq T} C_1^2 H^2\|\nabla p\|^2 (\mu_1+4)\frac{1-\exp[-\beta t]}{\beta}
\leq \frac{(\mu_1+4)}{\beta} C_1^2 H^2 \max_{0\leq t \leq T}\|\nabla p\|^2.
\end{gathered}
\end{equation}
Similarly,
\begin{equation}
\begin{gathered}
   \int_{0}^t \exp [-\beta (t-t')] 8\mu_1 \|\eta_p\|^2\, dt'\leq 8\frac{\mu_1}{\beta}\max_{0\leq t\leq T}\|\eta_p\|^2,\\
   \int_{0}^t \exp [-\beta (t-t')] \chi\|\eta\|^2\, dt'\leq  \frac{\chi}{\beta}\max_{0\leq t\leq T}\|\eta\|^2.
   \end{gathered}
\end{equation}
We simplify by the following estimates.
\begin{equation}
    \begin{gathered}
        \frac{4}{\nu}M^2 \int_{0}^t \|\nabla u\|^2 \|\nabla \eta\|^2\, dt' \leq  \frac{4}{\nu}M^2 \|\nabla u\|^2_{L^4(0,T;L^2(\Omega))}\|\nabla \eta\|^2_{L^4(0,T;L^2(\Omega))},\\
         \frac{4}{\nu}M^2 \int_{0}^t \|\nabla v^h\|^2 \|\nabla \eta\|^2\, dt' \leq  \frac{4}{\nu}M^2 \|\nabla v^h\|^2_{L^4(0,T;L^2(\Omega))}\|\nabla \eta\|^2_{L^4(0,T;L^2(\Omega))}.
    \end{gathered}
\end{equation}
Since $t>t'$, and $t\leq T$, $\exp [-\beta (t-t')]\leq 1$. Hence (\ref{error-estimate-after-IBP}) becomes
\allowdisplaybreaks
\begin{align}
&    \|\phi^h\|^2 +\frac{1}{c^2}\|\phi^h_p\|^2 +\frac{\nu}{8}\int_{0}^t \exp [-\beta (t-t')]\|\nabla \phi^h\|^2 \, dt' \nonumber\\
&\quad
\leq \exp [-\beta t] \big[\|\phi^h(0)\|^2 +\frac{1}{c^2}\|\phi^h_p(0)\|^2 \big]
+\frac{(\mu_1+4)}{\beta} C_1^2 H^2 \max_{0\leq t \leq T}\|\nabla p\|^2 \nonumber\\
&\quad
+8\frac{\mu_1}{\beta}\max_{0\leq t\leq T}\|\eta_p\|^2+\frac{\chi}{\beta}\max_{0\leq t\leq T}\|\eta\|^2+ 4C_1^2 H^2 \|\nabla \eta_p\|^2_{L^2(0,T;L^2(\Omega))}\nonumber\\
&\quad
+\frac{4}{\nu}M^2 \big[\|\nabla u\|^2_{L^4(0,T;L^2(\Omega))}+\|\nabla v^h\|^2_{L^4(0,T;L^2(\Omega))}\big]\|\nabla \eta\|^2_{L^4(0,T;L^2(\Omega))}\nonumber\\
&\quad
 + \frac{6}{\nu} \|\eta_p\|^2_{L^2(0,T;L^2(\Omega))}+\frac{16}{\mu_1 c^4}\|\eta_{p,t}\|^2_{L^2(0,T;L^2(\Omega))}
\nonumber\\
&\quad
 + \big[\frac{16}{\mu_1}+\frac{8}{\nu}+\frac{\nu}{3} \big] \|\nabla \eta\|^2_{L^2(0,T;L^2(\Omega))}
+\frac{8}{\nu}\int_{0}^T \|\eta_t\|^2_{-1}\, dt \nonumber\\
&\quad
+\frac{16}{\mu_1 c^4} \big[ \|q^h_t\|^2_{L^2(0,T;L^2(\Omega))}+ \frac{2}{3} \|\nabla u\|^4_{L^4(0,T;L^2(\Omega))} + \frac{2}{3} \|\nabla p\|^4_{L^4(0,T;L^3(\Omega))} \big].
\end{align}
Using the triangle inequality, we have the final result (\ref{error-estimate-result}).
\end{proof}

\subsection{Fully-discrete error estimates}\label{full-discrete-error}
Define the velocity error at $t_n$: $e^{n}_h=u(t_n)-v^{n}_h=\eta^n-\phi^n_h$, where $\eta^n = u(t_n)-\Tilde{v}$, $\phi^n_h= v^{n}_h-\Tilde{v}$, $\forall \Tilde{v}\in X^h$, the pressure error $e^{n}_{p,h}=p(t_n)-q^n_h =\eta^n_p-\phi^n_{p,h}$, where $\eta^n_p = p(t_n)-\Tilde{q}$, $\phi^{n}_{p,h}= q^n_h-\Tilde{q}$,  $\forall \Tilde{q}\in Q^h$.
\begin{theorem}
We assume 
\begin{equation}
\begin{gathered}
\chi- \frac{32C_2^4}{\nu^3}\max_{0\leq n\leq N-1}\|\nabla v^{n+1}_h\|^4\geq \alpha \chi,\quad 
\frac{\nu}{16}-\chi C_1^2 H^2\geq 0,
\text{and} \ \mu_1\geq \mu_2,
\label{fully-nudging-assumption}
\end{gathered}
\end{equation}
and the following regularity assumptions:
\begin{equation}\label{smoothAssum-fullydiscrete}
\begin{gathered}
u \in L^\infty(0,T;(H^{m+1}(\Omega))^3)
    \cap W^{1,4}(0,T;(H^1(\Omega))^3),\ 
 u_{tt} \in L^2(0,T;(H^{-1}(\Omega))^3),\\
p \in L^\infty(0,T;H^{m}(\Omega))
    \cap L^4(0,T;W^{1,3}(\Omega)),\
p_t \in L^2(0,T;H^{m}(\Omega)),\\
p_{tt} \in L^2(0,T;L^2(\Omega)).
\end{gathered}
\end{equation}
We have the following error estimate:
\allowdisplaybreaks

\begin{align}
&\frac{1}{2}\|e^{N}_h\|^2
+\frac{1}{2c^2}\|e^N_{p,h}\|^2
+\frac{\nu \Delta t}{16}
\sum_{n=0}^{N-1}(1+\beta \Delta t)^{n-N}
\|\nabla e^{n+1}_{h}\|^2
\leq
\exp\!\big[-\frac{\beta}{1+\beta \Delta t} T\big]\nonumber\\
&\bigg[
2\|e^{0}_h\|^2
+\frac{2}{c^2}\|e^{0}_{p,h}\|^2
+\frac{\nu}{2}\|\nabla e^{0}_h\|^2
+\chi\|e^{0}_h\|^2
+\frac{\nu}{2}\|\nabla \eta^{0}\|^2
+\chi\|\eta^{0}\|^2
\bigg]
\nonumber\\
&\quad
+2\|\eta^0\|^2
+\frac{2}{c^2}\|\eta^0_p\|^2
+\|\eta^{N}\|^2
+\frac{1}{c^2}\|\eta^N_{p}\|^2
+\chi \tn{\eta}^2_{l^2(L^2)}
\nonumber\\
&\quad
+\frac{(\mu_1+4)}{\beta}C_1^2H^2
\max_{0\le n\le N-1}\|\nabla p(t_{n+1})\|^2
+\Big[
\frac{8}{\nu}
+\frac{\nu}{3}
+\frac{\nu}{8}
+\frac{16}{\mu_1}
\Big]
\tn{\nabla\eta}^2_{l^2(L^2)}
\nonumber\\
&\quad
+\frac{4\Delta t^2}{3\nu}
\|u_{tt}\|^2_{L^2(0,T;H^{-1}(\Omega))}
+\frac{4}{\nu}
\|\eta_t\|^2_{L^2(0,T;H^{-1}(\Omega))}
+\frac{2M^2}{\nu}\Delta t^2
\tn{\nabla u}^4_{l^4(L^2)}
\nonumber\\
&\quad
+\frac{2M^2}{\nu}\Delta t^2
\|\nabla u_t\|^4_{L^4(0,T;L^2(\Omega))}
+\frac{32}{\mu_1c^4}
\|\eta_{p,t}\|^2_{L^2(0,T;H^{-1}(\Omega))}
+\frac{6}{\nu}
\tn{\eta_p}^2_{l^2(L^2)}
\nonumber\\
&\quad
+\frac{32\Delta t^2}{\mu_1c^4}
\|p_{tt}\|^2_{L^2(0,T;L^2(\Omega))}
+8\mu_1\tn{\eta_p}^2_{l^2(L^2)}+\frac{4M^2}{\nu}
\tn{\nabla v^h}^2_{l^4(L^2)}
\tn{\nabla\eta}^2_{l^4(L^2)}
\nonumber\\
&\quad
+\frac{32}{3c^4\mu_1}\tn{\nabla u}^4_{l^4(L^2)}
+\frac{32}{3c^4\mu_1}\tn{\nabla p}^4_{l^4(L^3)}+\frac{4C_2^2}{\nu}
\tn{u}_{l^\infty(L^2)}
\tn{\nabla u}_{l^2(L^2)}
\tn{\nabla\eta}^2_{l^4(L^2)},
\label{after_gronwall_2}
\end{align}
where $\beta=\min\{\frac{c^2\mu_1}{16}, \frac{\chi}{2}\}$.
\end{theorem}
\begin{corollary}
Under the regularity assumptions~(\ref{smoothAssum-fullydiscrete}), the parameter assumptions~(\ref{fully-nudging-assumption}), and the approximation properties~(\ref{prop}), the following estimate holds:
\allowdisplaybreaks
\begin{align}
&\frac{1}{2}\|e^{N}_h\|^2
+\frac{1}{2c^2}\|e^N_{p,h}\|^2
+\frac{\nu \Delta t}{16}
\sum_{n=0}^{N-1}(1+\beta \Delta t)^{n-N}
\|\nabla e^{n+1}_{h}\|^2
\nonumber\\
&\le
\exp\!\big[-\tfrac{\beta}{1+\beta\Delta t}T\big]
\left[
2\|e^{0}_h\|^2
+\frac{2}{c^2}\|e^{0}_{p,h}\|^2
+\frac{\nu}{2}\|\nabla e_h^0\|^2
+\chi\|e_h^0\|^2
\right]
\nonumber\\
&\quad
+C h^{2m+2}|u^0|_{H^{m+1}}^2
+\frac{C}{c^2}h^{2m}|p^0|_{H^m}^2
+C h^{2m+2}|u^N|_{H^{m+1}}^2
+\frac{C}{c^2}h^{2m}|p^N|_{H^m}^2
\nonumber\\
&\quad
+\exp\!\big[-\tfrac{\beta}{1+\beta\Delta t}T\big]
\left[
C\nu h^{2m}|u^0|_{H^{m+1}}^2
+C\chi h^{2m+2}|u^0|_{H^{m+1}}^2
\right]
\nonumber\\
&\quad
+C\chi h^{2m+2}\||u|_{H^{m+1}}\|_{l^2}^2
+\frac{(\mu_1+4)}{\beta}C_1^2H^2
\max_{0\le n\le N-1}\|\nabla p(t_{n+1})\|^2
\nonumber\\
&\quad
+C\Big[\frac{8}{\nu}+\frac{11\nu}{24}+\frac{16}{\mu_1}\Big]
h^{2m}\||u|_{H^{m+1}}\|_{l^2}^2
+\frac{4\Delta t^2}{3\nu}
\|u_{tt}\|^2_{L^2(0,T;H^{-1}(\Omega))}
\nonumber\\
&\quad
+\frac{C}{\nu}h^{2m+2}
|u_t|^2_{L^2(0,T;H^{m+1}(\Omega))}
+\frac{2M^2}{\nu}\Delta t^2
|\|\nabla u\||^4_{l^4(L^2)}
+\frac{2M^2}{\nu}\Delta t^2
\|\nabla u_t\|^4_{L^4(0,T;L^2(\Omega))}
\nonumber\\
&\quad
+\frac{C}{\mu_1c^4}h^{2m}
|p_t|^2_{L^2(0,T;H^m(\Omega))}
+\frac{C}{\nu}h^{2m}\||p|_{H^m}\|_{l^2}^2
+\frac{32\Delta t^2}{\mu_1c^4}
\|p_{tt}\|^2_{L^2(0,T;L^2(\Omega))}
\nonumber\\
&\quad
+C\mu_1 h^{2m}\||p|_{H^m}\|_{l^2}^2
+CC_1^2H^2h^{2m-2}\||p|_{H^m}\|_{l^2}^2
\nonumber\\
&\quad
+\frac{32}{3c^4\mu_1}
|\|\nabla u\||^4_{l^4(L^2)}
+\frac{CM^2}{\nu}
|\|\nabla v^h\||^2_{l^4(L^2)}
h^{2m}\||u|_{H^{m+1}}\|_{l^4}^2
\nonumber\\
&\quad
+\frac{32}{3c^4\mu_1}
|\|\nabla p\||^4_{l^4(L^3)}
+\frac{CC_2^2}{\nu}
|\|u\||_{l^\infty(L^2)}
|\|\nabla u\||_{l^2(L^2)}
h^{2m}\||u|_{H^{m+1}}\|_{l^4}^2,
\label{after_gronwall_approx}
\end{align}
where $\beta=\min\{\frac{c^2\mu_1}{16}, \frac{\chi}{2}\}$.
\end{corollary}

\begin{corollary}\label{corollary-fully-after-approx}
Under the regularity assumptions~(\ref{smoothAssum-fullydiscrete}), the parameter assumptions~(\ref{fully-nudging-assumption}), and the approximation properties~(\ref{prop}), the following estimate holds:
\begin{equation}
\begin{gathered}
    \frac{1}{2}\|e^{N}_h\|^2
    +\frac{1}{2c^2}\|e^N_{p,h}\|^2
    + \frac{\nu \Delta t}{16} \sum_{n=0}^{N-1}
(1+\beta \Delta t)^{n-N} \|\nabla e^{n+1}_{h}\|^2
\leq \exp\!
\big[-\frac{\beta}{1+\beta \Delta t}T\big]\\
\left[2\|e^{0}_h\|^2+\frac{2}{c^2}\|e^{0}_{p,h}\|^2+\frac{\nu}{2}\|\nabla e_h^0\|^2+\chi\|e_h^0\|^2+C\nu h^{2m}
+C\chi h^{2m+2}\right]\\
+ C\big[H^2+\Delta t^2+\frac{1}{c^4\mu_1}+\chi h^{2m+2}+
\mu_1 h^{2m}+ h^{2m} + H^2 h^{2m-2}
\big],
\label{after_gronwall_approx_after_nudging}
\end{gathered}
\end{equation}
where $\beta=\min\{\frac{c^2\mu_1}{16}, \frac{\chi}{2}\}$.
\end{corollary}
\begin{remark}\label{remark-fully}
The above estimate suggests choosing the pressure nudging parameter as
$\mu_1=\mathcal{O}(H^{-2})$, then
\[
\|e_h^N\|+\frac1c\|e_{p,h}^N\|
=\mathcal{O}\!\left(H+\Delta t+Hh^{m-1}+h^m
+\tfrac{h^{m}}{H}+\chi^{1/2}h^{m+1}\right).
\]
The last two contributions come from the terms $\mu_1 h^{2m}$ and $\chi h^{2m+2}$ in \eqref{after_gronwall_approx_after_nudging}: with
$\mu_1=\mathcal{O}(H^{-2})$ one has $\mu_1 h^{2m}=\mathcal{O}(h^{2m}/H^{2})$.
They are dominated by the leading $\mathcal{O}(H)$ term precisely when $h^{m}\lesssim H^{2}$ and $\chi\lesssim H^{2}h^{-2m-2}$, which is the regime of practical interest (the observation mesh is much coarser than the computational mesh).  Under these two conditions the estimate reduces to the announced
$\mathcal{O}(H+\Delta t+Hh^{m-1}+h^{m})$.

As in Remark \ref{remark-semi}, if $c^2\mu_1/16\leq \frac{\chi}{2}$, the corresponding pressure error term in the above estimate has the improved dependence $H/c$ rather than $H$.
\end{remark}

\begin{proof}
We evaluate \eqref{weak-compre-mom}--\eqref{weak-compre-conti} at $t_{n+1}$, and subtract the nudged system \eqref{BE_momentum} and \eqref{fully-conti} from \eqref{weak-compre-mom}--\eqref{weak-compre-conti}. We set $w^h=\phi^{n+1}_h$, and $\lambda^h=\phi^{n+1}_{p,h}$ and add the resulting equations:
\begin{equation}
\begin{split}
    \frac{1}{\Delta t} (\phi^{n+1}_h - \phi^n_h, \phi^{n+1}_h) +\nu \|\nabla \phi^{n+1}_h\|^2 +\frac{\nu}{3} \|\nabla \cdot \phi^{n+1}_h\|^2 +\chi (I_H(\phi^{n+1}_h),\phi^{n+1}_h)\\
    -\mu_1(I_H(p(t_{n+1}))-q^{n+1}_h, \phi^{n+1}_{p,h})-\mu_2 (I_H(q^{n+1}_{h})-q^{n+1}_h,\phi^{n+1}_{p,h})\\
    = \nu(\nabla \eta^{n+1},\nabla \phi^{n+1}_h)+
    \left(u_t(t_{n+1})-\frac{u(t_{n+1})-u(t_n)}{\Delta t}, \phi^{n+1}_h \right)\\
    + \frac{1}{\Delta t}(\eta^{n+1} - \eta^n, \phi^{n+1}_h) 
+b^*(u(t_{n+1}),u(t_{n+1}), \phi^{n+1}_h )- b^*(v^{n}_h, v^{n+1}_h, \phi^{n+1}_h)\\
+\frac{\nu}{3}(\nabla \cdot \eta^{n+1}, \nabla \cdot \phi^{n+1}_h)
-(\eta^{n+1}_p, \nabla \cdot \phi^{n+1}_h) +(\nabla \cdot \eta^{n+1}, \phi^{n+1}_{p,h})\\
    +\frac{1}{c^2}(\frac{\partial p}{\partial t}(t_{n+1})+ u(t_{n+1})\cdot \nabla p(t_{n+1}), \phi^{n+1}_{p,h}) + \chi (I_H(\eta^{n+1}), \phi^{n+1}_h).
    \end{split}
\end{equation}
The proof follows the same general structure as the semi-discrete error estimate in Section \ref{semi-discrete-error}. Thus, we focus on the terms that are new in the fully discrete setting.

Applying the polarization identity, we have
\begin{equation}
    \frac{1}{\Delta t} (\phi^{n+1}_h - \phi^n_h, \phi^{n+1}_h) =\frac{1}{2\Delta t}\left(\|\phi^{n+1}_h\|^2 - \|\phi^{n}_h\|^2+\|\phi^{n+1}_h-\phi^{n}_h\|^2 \right).
\end{equation}
\begin{equation}
\begin{gathered}
     (u_t(t_{n+1})-\frac{u(t_{n+1})-u(t_n)}{\Delta t}, \phi^{n+1}_h )
     \leq \|u_t(t_{n+1})-\frac{u(t_{n+1})-u(t_n)}{\Delta t}\|_{-1}\|\nabla \phi^{n+1}_h\|\\
     \leq \frac{2}{\nu}\|u_t(t_{n+1})-\frac{u(t_{n+1})-u(t_n)}{\Delta t}\|^2_{-1} + \frac{\nu}{8}\|\nabla \phi^{n+1}_h\|^2.
     \end{gathered}   
\end{equation}
Using a Taylor expansion with remainder term, we have
\begin{equation}
\begin{gathered}
    \|u_t(t_{n+1})-\frac{u(t_{n+1})-u(t_n)}{\Delta t} \|^2_{-1} =\frac{1}{\Delta t^2}\|\int_{t_n}^{t_{n+1}}  u_{tt} (s) (t_{n}-s)\, ds\|^2_{-1}\\
    \leq \frac{1}{\Delta t^2} \left(\int_{t_n}^{t_{n+1}} \|u_{tt}(s)\|_{-1}|t_n-s|\, ds\right)^2\\
 \leq  \frac{1}{\Delta t^2} \int_{t_n}^{t_{n+1}} (t_n-s)^2\,ds \int_{t_n}^{t_{n+1}}\|u_{tt}(s)\|^2_{-1}\,ds
 =\frac{\Delta t}{3}\|u_{tt}\|^2_{L^2(t_n, t_{n+1}; H^{-1} (\Omega))}.\label{u_t-term}
    \end{gathered}
\end{equation}
Thus,
\begin{equation}
(u_t(t_{n+1})-\frac{u(t_{n+1})-u(t_n)}{\Delta t}, \phi^{n+1}_h ) \leq \frac{2\Delta t}{3\nu} \|u_{tt}\|^2_{L^2(t_n, t_{n+1}; H^{-1}(\Omega))} +\frac{\nu}{8}\|\nabla \phi^{n+1}_h\|^2.
\end{equation}
For $\frac{1}{\Delta t}(\eta^{n+1} - \eta^n, \phi^{n+1}_h)$, we have
\begin{equation}
\begin{split}
\frac{1}{\Delta t}(\eta^{n+1} - \eta^n, \phi^{n+1}_h)
\leq \frac{\nu}{8}\|\nabla \phi^{n+1}_h\|^2 +\frac{2}{ \nu \Delta t} \int_{t_n}^{t_{n+1}} \|\eta_t\|^2_{-1}\, dt.
\label{eta_t-term}
\end{split}
\end{equation}
For the nonlinear terms, we have
\begin{equation}
\begin{gathered}
    b^*(u(t_{n+1}),u(t_{n+1}), \phi^{n+1}_h )- b^*(v^{n}_h, v^{n+1}_h, \phi^{n+1}_h)\\
    =b^*(u(t_{n+1})-u(t_n), u(t_{n+1}),\phi^{n+1}_h) +b^*(\eta^n, v^{n+1}_h, \phi^{n+1}_h)\\
    +b^*(u(t_n),\eta^{n+1}, \phi^{n+1}_h) + b^*(\phi^{n}_h, v^{n+1}_h, \phi^{n+1}_h).
    \end{gathered}
\end{equation}
We will bound the above terms as follows:
\begin{equation}
\begin{gathered}
    b^*(u(t_{n+1})-u(t_n), u(t_{n+1}),\phi^{n+1}_h) \\
    \leq M\|\nabla (u(t_{n+1})-u(t_n))\|\|\nabla u(t_{n+1})\| \|\nabla \phi^{n+1}_h\|\\
    \leq \frac{\nu}{8}\|\nabla \phi^{n+1}_h\|^2 + \frac{2M^2}{\nu}\|\nabla (u(t_{n+1})-u(t_n))\|^2\|\nabla u(t_{n+1})\|^2\\
    \leq \frac{\nu}{8}\|\nabla \phi^{n+1}_h\|^2 + \frac{2M^2}{\nu}\Delta t \int_{t_n}^{t_{n+1}} \|\nabla u_t\|^2\, ds \|\nabla u(t_{n+1})\|^2\\
    \leq \frac{\nu}{8}\|\nabla \phi^{n+1}_h\|^2  + \frac{M^2}{\nu} \Delta t \|\nabla u_t\|^4_{L4(t_n, t_{n+1}; L^2(\Omega))}
    + \frac{M^2}{\nu} \Delta t^2 \|\nabla u(t_{n+1})\|^4,\label{non_linear_1}
    \end{gathered}
\end{equation}
\begin{equation}
\begin{gathered}
    b^*(\eta^n, v^{n+1}_h, \phi^{n+1}_h)\leq M\|\nabla \eta^n\|\|\nabla v^{n+1}_h\|\|\nabla \phi^{n+1}_h\|\\
    \leq \frac{2M^2}{\nu}\|\nabla \eta^n\|^2 \|\nabla v^{n+1}_h\|^2 +\frac{\nu}{8}\|\nabla \phi^{n+1}_h\|^2,
    \end{gathered}
\end{equation}
\begin{equation}
\begin{gathered}
    b^*(u(t_n),\eta^{n+1}, \phi^{n+1}_h) \leq C_2\|\nabla u(t_n)\|^{1/2} \| u(t_n)\|^{1/2}\|\nabla \eta^{n+1}\|\|\nabla \phi^{n+1}_h\|\\
    \leq \frac{\nu}{8}\|\nabla \phi^{n+1}_h\|^2 +\frac{2C_2^2}{\nu}\|\nabla u(t_n)\|\|u(t_n)\|\|\nabla \eta^{n+1}\|^2,
    \end{gathered}
\end{equation}
\begin{equation}
\begin{gathered}
    b^*(\phi^{n}_h, v^{n+1}_h, \phi^{n+1}_h)\leq C_2 \|\nabla \phi^n_h\|^{1/2}\|\nabla \phi^{n}_h\|^{1/2}\|
    \nabla v^{n+1}_h\|\|\nabla \phi^{n+1}_h\| \\
    \leq \frac{\nu}{8}\|\nabla \phi^{n+1}_h\|^2 + \frac{2C_2^2}{\nu}\|\nabla \phi^n_h\|\| \phi^n_h\|\|\nabla v^{n+1}_h\|^2\\
    \leq \frac{\nu}{8}\|\nabla \phi^{n+1}_h\|^2 + \frac{\nu}{8}\|\nabla \phi^n_h\|^2 +\frac{8C_2^4}{\nu^3}\|\nabla v^{n+1}_h\|^4 \|\phi^n_h\|^2.\label{non_linear_4}
    \end{gathered}
\end{equation}

\begin{equation}
\begin{gathered}
    \frac{1}{c^2}(\frac{\partial p}{\partial t}(t_{n+1}), \phi^{n+1}_{p,h})
    =\frac{1}{c^2} (\frac{\eta^{n+1}_{p}-\eta^{n}_{p}}{\Delta t}, \phi^{n+1}_{p,h}) -\frac{1}{c^2\Delta t}(\phi^{n+1}_{p,h}-\phi^{n}_{p,h}, \phi^{n+1}_{p,h})\\
     + \frac{1}{c^2} (\frac{\partial p}{\partial t}(t_{n+1}) -\frac{p(t_{n+1})-p(t_n)}{\Delta t},\phi^{n+1}_{p,h})\\
     \leq  \frac{16}{\mu_1 c^4 \Delta t} \int_{t_n}^{t_{n+1}} \|\eta_{p,t}\|^2_{-1}\, dt-\frac{1}{c^2 2\Delta t}(\|\phi^{n+1}_{p,h}\|^2-\|\phi^{n}_{p,h}\|^2+\|\phi^{n+1}_{p,h}-\phi^{n}_{p,h}\|^2)\\
     +\frac{\mu_1}{64}\|\phi^{n+1}_{p,h}\|^2 
     +\frac{16}{\mu_1c^4} \Delta t\|p_{tt}\|^2_{L^2(t_n,t_{n+1}; L^2(\Omega))} +\frac{\mu_1}{64}\|\phi^{n+1}_{p,h}\|^2.\label{p_t_term}
    \end{gathered}
\end{equation}
Assume $\mu_1\geq \mu_2$. Since $I_H$ is an $L^2$-projection, the estimate follows:
\begin{equation}
\begin{gathered}
    -\mu_1 (I_H(p(t_{n+1})-q^{n+1}_{h}), \phi^{n+1}_{p,h})-\mu_2 (I_H(q^{n+1}_h)-q^{n+1}_h,\phi^{n+1}_{p,h})\\
    \geq  -\frac{\mu_1}{2}C_1^2 H^2 \|\nabla p(t_{n+1})\|^2 +  \frac{\mu_1}{8} \|\phi^{n+1}_{p,h}\|^2 - 4\mu_1 \|\eta^{n+1}_p\|^2
\\
-\frac{2 (\mu_1-\mu_2)}{\mu_1} C_1^2 H^2 (\|\nabla \eta^{n+1}_p\|^2+\|\nabla p(t_{n+1})\|^2).\label{estimate-pressure-fully}
    \end{gathered}
    \end{equation}
For the following terms, we have
\begin{equation}
\begin{gathered}
     (\nabla \cdot \eta^{n+1}, \phi^{n+1}_{p,h})\leq \frac{\mu_1}{32}\|\phi^{n+1}_{p,h}\|^2 + \frac{8}{\mu_1} \|\nabla \eta^{n+1}\|^2,\\
    -(\eta^{n+1}_p, \nabla \cdot \phi^{n+1}_h)\leq \frac{3}{\nu}\|\eta^{n+1}_p\|^2 +\frac{\nu}{12}\|\nabla \cdot \phi^{n+1}_h\|^2.
    \end{gathered}
\end{equation}
Denote $A_0:=\frac{1}{c^2}\left(u(t_{n+1})\cdot \nabla p(t_{n+1}), \phi^{n+1}_{p,h} \right)$. We bound $A_0$ similarly as the semi-discrete analysis:
\begin{equation}
\begin{gathered}
A_0 \leq \frac{16}{3c^4\mu_1} \big[\|\nabla u(t_{n+1})\|^4 + \|\nabla p(t_{n+1})\|^4_{L^3}\big] + \frac{\mu_1}{32}\|\phi^{n+1}_{p,h}\|^2.\label{estimate-pressure-2-fully}
    \end{gathered}
\end{equation}
For $\nu (\nabla \eta^{n+1},\nabla \phi^{n+1}_h)$ and $\frac{\nu}{3}(\nabla \cdot \eta^{n+1}, \nabla \cdot \phi^{n+1}_h)$ terms, we have
\begin{equation}
\begin{gathered}
    \nu (\nabla \eta^{n+1},\nabla \phi^{n+1}_h)\leq \frac{\nu}{16}\|\nabla \phi^{n+1}_h\|^2 + \frac{4}{\nu}\|\nabla \eta^{n+1}\|^2,\\
    \frac{\nu}{3}(\nabla \cdot \eta^{n+1},
    \nabla \cdot \phi^{n+1}_h)\leq \frac{\nu}{6}\|\nabla  \eta^{n+1}\|^2 +  \frac{\nu}{6}\|\nabla \cdot \phi^{n+1}_h\|^2.
    \end{gathered}
\end{equation}
Lastly, we handle the nudging terms for velocity.
\begin{equation}
\begin{gathered}
    \chi (I_H(\phi^{n+1}_h), \phi^{n+1}_h)
    \geq \chi\|\phi^{n+1}_h\|^2 -\chi C_1^2 H^2 \|\nabla \phi^{n+1}_h\|^2,\\
    \chi (I_H(\eta^{n+1}), \phi^{n+1}_h) \leq \frac{\chi}{2}\|I_H(\eta^{n+1})\|^2 +\frac{\chi}{2}\|\phi^{n+1}_h\|^2.\label{estimate-velocity-nudging-fully}
    \end{gathered}
\end{equation}
Adding inequalities (\ref{u_t-term}), (\ref{eta_t-term}), (\ref{non_linear_1})--(\ref{estimate-velocity-nudging-fully}), we obtain
\allowdisplaybreaks

\begin{align}
&\frac{1}{2\Delta t}
\left(
\|\phi^{n+1}_h\|^2-\|\phi^{n}_h\|^2
+\|\phi^{n+1}_h-\phi^{n}_h\|^2
\right)
+\frac{\nu}{12}\|\nabla\cdot\phi^{n+1}_h\|^2
\nonumber\\
&+\frac{1}{2\Delta t}
\left(
\frac{1}{c^2}\|\phi^{n+1}_{p,h}\|^2
-\frac{1}{c^2}\|\phi^{n}_{p,h}\|^2
+\frac{1}{c^2}\|\phi^{n+1}_{p,h}-\phi^{n}_{p,h}\|^2
\right)
+\frac{\mu_1}{32}\|\phi^{n+1}_{p,h}\|^2
\nonumber\\
&+\frac{\nu}{16}\|\nabla\phi^{n+1}_h\|^2
+\frac{\nu}{8}
\left(
\|\nabla\phi^{n+1}_h\|^2
-\|\nabla\phi^{n}_h\|^2
\right)
+\left(
\frac{\nu}{16}-\chi C_1^2H^2
\right)
\|\nabla\phi^{n+1}_h\|^2
\nonumber\\
&+\frac{\chi}{4}\|\phi^{n+1}_h\|^2
+\frac{\chi}{4}
\left(
\|\phi^{n+1}_h\|^2-\|\phi^n_h\|^2
\right)
+\left(
\frac{\chi}{4}
-\frac{8C_2^4}{\nu^3}\|\nabla v^{n+1}_h\|^4
\right)
\|\phi^n_h\|^2
\nonumber\\
&\le
\frac{\chi}{2}\|I_H(\eta^{n+1})\|^2
+\left(
\frac{4}{\nu}
+\frac{\nu}{6}
+\frac{8}{\mu_1}
\right)
\|\nabla\eta^{n+1}\|^2
+\frac{2\Delta t}{3\nu}
\|u_{tt}\|^2_{L^2(t_n,t_{n+1};H^{-1} (\Omega))}
\nonumber\\
&+\frac{2}{\nu\Delta t}
\int_{t_n}^{t_{n+1}}\|\eta_t\|_{-1}^2\,dt
+\frac{M^2}{\nu}\Delta t
\|\nabla u_t\|^4_{L^4(t_n,t_{n+1};L^2(\Omega))}
+\frac{M^2}{\nu}\Delta t^2
\|\nabla u(t_{n+1})\|^4
\nonumber\\
&+\frac{2M^2}{\nu}
\|\nabla\eta^n\|^2
\|\nabla v_h^{n+1}\|^2
+\frac{2C_2^2}{\nu}
\|\nabla u(t_n)\|
\|u(t_n)\|
\|\nabla\eta^{n+1}\|^2
\nonumber\\
&+\frac{16}{\mu_1c^4\Delta t}
\int_{t_n}^{t_{n+1}}
\|\eta_{p,t}\|_{-1}^2\,dt
+\frac{16\Delta t}{\mu_1c^4}
\|p_{tt}\|^2_{L^2(t_n,t_{n+1};L^2(\Omega))}
+4\mu_1\|\eta^{n+1}_p\|^2
\nonumber\\
&+\frac{\mu_1}{2}C_1^2H^2
\|\nabla p(t_{n+1})\|^2
+2C_1^2H^2
\left(
\|\nabla\eta^{n+1}_p\|^2
+\|\nabla p(t_{n+1})\|^2
\right)
\nonumber\\
&+\frac{3}{\nu}\|\eta^{n+1}_p\|^2
+\frac{16}{3c^4\mu_1}
\left(
\|\nabla u(t_{n+1})\|^4
+\|\nabla p(t_{n+1})\|^4_{L^3}
\right).
\label{before-sum}
\end{align}
Applying the assumptions on the nudging parameters (\ref{fully-nudging-assumption}), and multiplying it by $2$, 
\allowdisplaybreaks
\begin{align}
&\frac{1}{\Delta t} \Big[\|\phi^{n+1}_h\|^2
+\frac{1}{c^2}\|\phi^{n+1}_{p,h}\|^2
-\|\phi^{n}_h\|^2
-\frac{1}{c^2}\|\phi^{n}_{p,h}\|^2\Big]
+\frac{\mu_1}{16}\|\phi^{n+1}_{p,h}\|^2
\nonumber\\
&+\frac{1}{\Delta t}
\left(
\|\phi^{n+1}_h-\phi^{n}_h\|^2
+\frac{1}{c^2}\|\phi^{n+1}_{p,h}-\phi^{n}_{p,h}\|^2
\right)
+\frac{\nu}{8}\|\nabla \phi^{n+1}_h\|^2
+\frac{\chi}{2}\|\phi^{n+1}_h\|^2
\nonumber\\
&+\frac{\nu}{4}
\left(
\|\nabla \phi^{n+1}_h\|^2
-\|\nabla \phi^{n}_h\|^2
\right)
+\frac{\chi}{2}
\left(
\|\phi^{n+1}_h\|^2
-\|\phi^n_h\|^2
\right)
+\frac{\alpha\chi}{2}\|\phi^n_h\|^2
\nonumber\\
&\le
\chi\|I_H(\eta^{n+1})\|^2
+\left(
\frac{8}{\nu}
+\frac{\nu}{3}
+\frac{16}{\mu_1}
\right)
\|\nabla\eta^{n+1}\|^2
+\frac{4\Delta t}{3\nu}
\|u_{tt}\|^2_{L^2(t_n,t_{n+1};H^{-1}(\Omega))}
\nonumber\\
&+\frac{4}{\nu\Delta t}
\int_{t_n}^{t_{n+1}}\|\eta_t\|_{-1}^2\,dt
+\frac{2M^2}{\nu}\Delta t
\|\nabla u_t\|^4_{L^4(t_n,t_{n+1};L^2(\Omega))}
+\frac{2M^2}{\nu}\Delta t^2
\|\nabla u(t_{n+1})\|^4
\nonumber\\
&+\frac{4M^2}{\nu}
\|\nabla\eta^n\|^2
\|\nabla v_h^{n+1}\|^2
+\frac{4C_2^2}{\nu}
\|\nabla u(t_n)\|
\|u(t_n)\|
\|\nabla\eta^{n+1}\|^2
\nonumber\\
&+\frac{32}{\mu_1c^4\Delta t}
\int_{t_n}^{t_{n+1}}
\|\eta_{p,t}\|_{-1}^2\,dt
+\frac{32\Delta t}{\mu_1c^4}
\|p_{tt}\|^2_{L^2(t_n,t_{n+1};L^2(\Omega))}
+8\mu_1\|\eta_p^{n+1}\|^2
\nonumber\\
&+(\mu_1+4)C_1^2H^2
\|\nabla p(t_{n+1})\|^2
+4C_1^2H^2 \|\nabla\eta_p^{n+1}\|^2
\nonumber\\
&+\frac{6}{\nu}\|\eta_p^{n+1}\|^2
+\frac{32}{3c^4\mu_1}
\left(
\|\nabla u(t_{n+1})\|^4
+\|\nabla p(t_{n+1})\|_{L^3}^4
\right).
\label{before-sum-after-assumps}
\end{align}

By a discrete Gronwall inequality  for the backward difference form \cite{emmrich1999discrete}, we have

\allowdisplaybreaks

\begin{align}
&\|\phi^{N}_h\|^2 +\frac{1}{c^2}\|\phi^{N}_{p,h}\|^2+ \frac{\nu \Delta t}{8} \sum_{n=0}^{N-1}
(1+\beta \Delta t)^{n-N} \|\nabla \phi^{n+1}_{h}\|^2\nonumber\\
&+ \sum_{n=0}^{N-1}
(1+\beta \Delta t)^{n-N} (\|\phi^{n+1}_h-\phi^{n}_h\|^2+\frac{1}{c^2}\|\phi^{n+1}_{p,h}-\phi^{n}_{p,h}\|^2)\nonumber\\
&+\Delta t \sum_{n=0}^{N-1} (1+\beta \Delta t)^{n-N} (\frac{\nu}{4}(\|\nabla \phi^{n+1}_h\|^2 - \|\nabla \phi^{n}_h\|^2)+\frac{\chi}{2} (\|\phi^{n+1}_h\|^2-\|\phi^n_{h}\|^2) )\nonumber\\
&\leq (1+\beta \Delta t)^{-N} (\|\phi^{0}_h\|^2+\frac{1}{c^2}\|\phi^{0}_{p,h}\|^2)\nonumber\\
&+(\mu_1+4) C_1^2 H^2 \max_{0\leq n\leq N-1} \|\nabla p(t_{n+1})\|^2
\Delta t \sum_{n=0}^{N-1} (1+\beta \Delta t)^{n-N}\nonumber\\
&+\Delta t \sum_{n=0}^{N-1} (1+\beta \Delta t)^{n-N}\bigg[
 \chi\|I_H(\eta^{n+1})\|^2 +(\frac{8}{\nu}+\frac{\nu}{3} +\frac{16}{\mu_1})\|\nabla \eta^{n+1}\|^2\nonumber\\
&\qquad + \frac{4\Delta t}{3\nu} \|u_{tt}\|^2_{L^2(t_n, t_{n+1}; H^{-1}(\Omega))} +\frac{4}{ \nu \Delta t} \|\eta_t\|^2_{L^2(t_n, t_{n+1}; H^{-1} (\Omega))} + \frac{2M^2}{\nu} \Delta t^2 \|\nabla u(t_{n+1})\|^4\nonumber\\
&\qquad + \frac{2M^2}{\nu} \Delta t \|\nabla u_t\|^4_{L^4(t_n, t_{n+1}; L^2(\Omega))}
    +\frac{4M^2}{\nu}\|\nabla \eta^n\|^2 \|\nabla v^{n+1}_h\|^2 \nonumber\\
&\qquad+\frac{4C_2^2}{\nu}\|\nabla u(t_n)\|\|u(t_n)\|\|\nabla \eta^{n+1}\|^2
    +\frac{32}{\mu_1 c^4 \Delta t} \|\eta_{p,t}\|^2_{L^2(t_n, t_{n+1}; H^{-1} (\Omega))} \nonumber\\
&\qquad+ \frac{32\Delta t}{\mu_1c^4} \|p_{tt}\|^2_{L^2(t_n,t_{n+1}; L^2(\Omega))}
+ 4 C_1^2 H^2\|\nabla \eta^{n+1}_p\|^2+ 8\mu_1 \|\eta^{n+1}_p\|^2
\nonumber\\
&\qquad
+\frac{6}{\nu}\|\eta^{n+1}_p\|^2 +\frac{32}{3c^4\mu_1} \left(\|\nabla u(t_{n+1})\|^4 + \|\nabla p(t_{n+1})\|^4_{L^3}\right)
    \bigg].
\label{after_gronwall_1}
\end{align}
Denote $A_1:=\Delta t \sum_{n=0}^{N-1}(1+\beta \Delta t)^{n-N} (\frac{\nu}{4}(\|\nabla \phi^{n+1}_h\|^2 - \|\nabla \phi^{n}_h\|^2)+\frac{\chi}{2} (\|\phi^{n+1}_h\|^2-\|\phi^n_{h}\|^2))$.
\begin{equation}
\begin{gathered}
    A_1=\frac{1}{1+\beta \Delta t} (\frac{\nu}{4}\|\nabla \phi^{N}_h\|^2 + \frac{\chi}{2}\|\phi^{N}_h\|^2) -(1+\beta \Delta t)^{-N} (\frac{\nu}{4}\|\nabla \phi^{0}_h\|^2 + \frac{\chi}{2}\|\phi^{0}_h\|^2)\\
    +\sum_{n=1}^{N-1} (1-\frac{1}{1+\beta \Delta t}) (1+\beta \Delta t)^{n-N} (\frac{\nu}{4}\|\nabla \phi^{n}_h\|^2 + \frac{\chi}{2}\|\phi^{n}_h\|^2).
    \end{gathered}
\end{equation}
Thus, only the initial negative contribution needs to be bounded. Since
\begin{equation}
\begin{gathered}
    (1+\beta \Delta t)^{-N}\leq \exp[-\frac{\beta}{1+\beta \Delta t} T],
    \end{gathered}
\end{equation}
\begin{equation}
    (1+\beta \Delta t)^{-N} (\frac{\nu}{4}\|\nabla \phi^{0}_h\|^2 + \frac{\chi}{2}\|\phi^{0}_h\|^2)\leq \exp[-\frac{\beta}{1+\beta \Delta t} T](\frac{\nu}{4}\|\nabla \phi^{0}_h\|^2 + \frac{\chi}{2}\|\phi^{0}_h\|^2).
\end{equation}
Denote $A_2:=
(\mu_1+4) C_1^2 H^2 \max_{0\leq n\leq N-1} \|\nabla p(t_{n+1})\|^2   \Delta t \sum_{n=0}^{N-1} (1+\beta \Delta t)^{n-N}$. Since
\begin{equation}
    \sum_{n=0}^{N-1} (1+\beta \Delta t)^{n-N} = \frac{1-(1+\beta \Delta t )^{-N}}{\beta \Delta t }\leq \frac{1}{\beta \Delta t},
\end{equation}
\begin{equation}
    A_2\leq \frac{(\mu_1+4)}{\beta} C_1^2 H^2 \max_{0\leq n\leq N-1} \|\nabla p(t_{n+1})\|^2.
\end{equation}
For the rest of the terms on the right hand side of (\ref{after_gronwall_1}), we use $(1+\beta \Delta t)^{n-N} \leq 1$.
\begin{equation}
    \begin{gathered}
   \Delta t \sum_{n=0}^{N-1} 
    \frac{4M^2}{\nu}\|\nabla \eta^n\|^2 \|\nabla v^{n+1}_h\|^2 \leq \frac{4M^2}{\nu} |\|\nabla v^h\||^2_{l^4(L^2)}|\|\nabla \eta\||^2_{l^4(L^2)}.
    \end{gathered}
\end{equation}
Denote 
 $A_3:=\Delta t \sum_{n=0}^{N-1} \frac{4C_2^2}{\nu}\|\nabla u(t_n)\|\|u(t_n)\|\|\nabla \eta^{n+1}\|^2$. We have 
\begin{equation}
\begin{split}
    A_3 \leq  \frac{4C_2^2}{\nu} |\|u\||_{l^\infty(L^2)} |\|\nabla u\||_{l^2(L^2)} |\|\nabla \eta\||^2_{l^4(L^2)}.
    \end{split}
\end{equation}
Hence, (\ref{after_gronwall_1}) becomes
\begin{equation}
\begin{gathered}
    \|\phi^{N}_h\|^2 +\frac{1}{c^2}\|\phi^{N}_{p,h}\|^2+ \frac{\nu \Delta t}{8} \sum_{n=0}^{N-1}
(1+\beta \Delta t)^{n-N} \|\nabla \phi^{n+1}_{h}\|^2
\leq  \exp[-\frac{\beta}{1+\beta \Delta t} T] \\ \big[\|\phi^{0}_h\|^2+\frac{1}{c^2}\|\phi^{0}_{p,h}\|^2+\frac{\nu}{4}\|\nabla \phi^{0}_h\|^2 + \frac{\chi}{2}\|\phi^{0}_h\|^2\big]
+ \frac{4\Delta t^2}{3\nu} \|u_{tt}\|^2_{L^2(0, T; H^{-1} (\Omega))}\\
+\chi|\|I_H(\eta)\||^2_{l^2(L^2)}
+\frac{(\mu_1+4)}{\beta} C_1^2 H^2 \max_{0\leq n\leq N-1} \|\nabla p(t_{n+1})\|^2
+\frac{4}{ \nu} \|\eta_t\|^2_{L^2(0, T; H^{-1} (\Omega))} \\
+(\frac{8}{\nu}+\frac{\nu}{3} +\frac{16}{\mu_1})|\|\nabla \eta\||^2_{l^2(L^2)}+ \frac{2M^2}{\nu} \Delta t^2 \big[|\|\nabla u\||^4_{l^4(L^2)}+   \|\nabla u_t\|^4_{L^4(0, T; L^2(\Omega))}\big]\\
 +\frac{32}{\mu_1 c^4} \|\eta_{p,t}\|^2_{L^2(0, T; H^{-1} (\Omega))}  + \frac{32\Delta t^2}{\mu_1c^4} \|p_{tt}\|^2_{L^2(0,T; L^2(\Omega))}
 + (8\mu_1+\frac{6}{\nu}) |\|\eta_p\||^2_{l^2(L^2)} \\
 +4 C_1^2 H^2|\|\nabla \eta_p\||^2_{l^2(L^2)}
+\frac{4M^2}{\nu} |\|\nabla v^h\||^2_{l^4(L^2)}|\|\nabla \eta\||^2_{l^4(L^2)}
\\
+\frac{32}{3c^4\mu_1} \big[|\|\nabla u\||^4_{l^4(L^2)} +|\|\nabla p\||^4_{l^4(L^3)} \big]
+\frac{4C_2^2}{\nu} |\|u\||_{l^\infty(L^2)} |\|\nabla u\||_{l^2(L^2)} |\|\nabla \eta\||^2_{l^4(L^2)}
.
\end{gathered}
\end{equation}
Applying the triangle inequality, we obtain the final result. 
\end{proof}

\section{Numerical tests}\label{numerical-tests}

With the pressure-nudging parameter at its optimal scaling $\mu_1=\mathcal{O}(H^{-2})$, the fully discrete backward Euler estimate of Section~\ref{error-estimates} Corollary \ref{corollary-fully-after-approx} and Remark \ref{remark-fully} reduces, at $t_N=T$, to
\begin{equation}\label{err-decomp-num}
		\|e_h^N\| + \tfrac{1}{c}\,\|e_{p,h}^N\|
		\;\lesssim\;
		e^{-\beta t_N}\Big(\|e_h^0\|+\tfrac1c\|e_{p,h}^0\|\Big)
		\;+\; C_H\, H
		\;+\; C_{\Delta t}\, \Delta t
		\;+\; C_h\, h^{m},
\end{equation}
where $C_H$, $C_{\Delta t}$, $C_h$ are positive constants, $\beta=\min\{\mu_1/16,\chi/2\}$, $m=2$ for Taylor--Hood $\mathbb{P}_2$--$\mathbb{P}_1$ ($\mathcal{O}(h^2)$ pressure, $\mathcal{O}(h^3)$ velocity in $L^2$), under assumptions~\eqref{fully-nudging-assumption} (which include $\mu_1\ge\mu_2$). All numerical tests below use the backward Euler method.

\paragraph{Setup.} FreeFEM~\cite{Hecht12} (v4.13), Taylor--Hood plain triangulations of $\Omega=(0,1)^2$. Accuracy is verified against the manufactured fields
\begin{equation}
\begin{gathered}
u_1 = \cos t\,\sin(\pi x)\sin(\pi y),\
u_2 = \sin t\,\sin(\pi x)\sin(\pi y),\\
p = \cos t\,\cos(\pi x)\cos(\pi y).
\end{gathered}
\end{equation}
Substituting $u_1,u_2,p$ into
\eqref{compressible_momentum}--\eqref{compressible_continuity} yields the body force $f$ and the continuity source
$g=\tfrac{1}{c^2}p_t+\nabla\cdot u$; that is, for the manufactured solution the right-hand side of \eqref{compressible_continuity} is replaced by $g$ rather than by zero. We impose no-slip boundary conditions on the nudged system, and the nudged flow is initially at rest at $t=0$. $I_H$ is the $L^2$-projection onto piecewise constants on a \emph{genuinely coarser} independent mesh $\mathcal{T}^H$ ($H>h$, never $H=h$). Spatial and parameter studies (Sections~\ref{sec:num-spatial}, \ref{sec:num-floor}, and \ref{sec:num-mu2}) compute the errors with respect to the manufactured exact solution. The temporal study (Section~\ref{sec:num-temporal}) computes the convergence rate using the successive-difference estimator. Unless otherwise stated, $c=1$, $\nu=1$, $\chi=100$, $\mu_1=\mu_2=100$. 
	
\subsection{Spatial convergence and space--time balance}\label{sec:num-spatial}\label{sec:num-balance}
Table~\ref{tab:spatial-balance}(a) fixes $\Delta t=0.005$ and refines $h$
($n_H=8$, $T=1.5$). The pressure converges at the expected second-order rate
throughout. The velocity initially exhibits super-convergence, with a rate of $2.34$, but the rate drops to $0.38$ and then $0.01$ as the mesh is further refined. This occurs because the $\mathcal{O}(h^3)$ error becomes smaller than the fixed
time-discretization error. Consequently, the
$C_{\Delta t}\Delta t$ term in \eqref{err-decomp-num} dominates the total error, rather than a limitation of the spatial discretization. These results show that the observed error floor is determined by $\Delta t$ rather than the final time $T$.

Table~\ref{tab:spatial-balance}(b) removes this confound by balancing $\Delta t=h^2$ ($n_H=8$, $T=1.5$): the pressure rate is
$2$ at every refinement, and the velocity rate steadily decreases toward $2$ as the $\mathcal{O}(\Delta t)=\mathcal{O}(h^2)$ term overtakes the $\mathcal{O}(h^3)$ term --- confirming that a first-order-in-time scheme delivers globally second-order accuracy once $\Delta t\sim h^2$. 
\begin{table}[H]
	\centering
	\resizebox{0.85\textwidth}{!}{%
	\begin{minipage}{\textwidth}
		\centering \small
		\setlength{\tabcolsep}{2.pt}
		\begin{minipage}{0.4\linewidth}
			\centering\textbf{(a) Fixed $\Delta t=0.005$}\\[3pt]
			\begin{tabular}{cccccc}
				\hline
				$n_h$ & $h$ & $\|u(T)-v^N_h\|$ & rate & $\|p(T)-q^N_h\|$ & rate \\
				\hline
				$8$  & $0.1250$ & $2.28\!\times\!10^{-4}$ & $-$   & $7.72\!\times\!10^{-4}$ & $-$   \\
				$16$ & $0.0625$ & $4.51\!\times\!10^{-5}$ & 2.34 & $1.92\!\times\!10^{-4}$ & 2.01 \\
				$32$ & $0.0312$ & $3.48\!\times\!10^{-5}$ & 0.38 & $4.77\!\times\!10^{-5}$ & 2.01 \\
				$48$ & $0.0208$ & $3.46\!\times\!10^{-5}$ & 0.01 & $2.11\!\times\!10^{-5}$ & 2.01 \\
				\hline
			\end{tabular}
		\end{minipage}%
		\hfill
		\begin{minipage}{0.45\linewidth}
			\centering\textbf{(b) Balanced $\Delta t=h^2$}\\[3pt]
			\begin{tabular}{cccccc}
				\hline
				$n_h$ & $h$ & $\|u(T)-v^N_h\|$ & rate & $\|p(T)-q^N_h\|$ & rate \\
				\hline
				$8$  & $0.1250$ & $2.75\!\times\!10^{-4}$ & $-$   & $8.95\!\times\!10^{-3}$ & $-$   \\
				$16$ & $0.0625$ & $4.24\!\times\!10^{-5}$ & 2.69 & $2.20\!\times\!10^{-3}$ & 2.03 \\
				$24$ & $0.0417$ & $1.56\!\times\!10^{-5}$ & 2.47 & $9.74\!\times\!10^{-4}$ & 2.01 \\
				$32$ & $0.0312$ & $8.03\!\times\!10^{-6}$ & 2.31 & $5.47\!\times\!10^{-4}$ & 2.00 \\
				$48$ & $0.0208$ & $3.31\!\times\!10^{-6}$ & 2.19 & $2.43\!\times\!10^{-4}$ & 2.00 \\
				\hline
			\end{tabular}
		\end{minipage}
	\end{minipage}%
	}
	\caption{$\chi=100$, $\mu_1=\mu_2=100$, $\nu=1$, $c=1$. (a) fixed $\Delta t=0.005$, $T=1.5$: velocity hits the temporal floor at $n_h\ge32$. (b) $\Delta t=h^2$, $T=1.5$: pressure exactly second order; velocity trending to second order.}
	\label{tab:spatial-balance}
\end{table}
	
\subsection{Temporal order}\label{sec:num-temporal}

Since the $\mathcal{O}(H)$ error contribution is $\Delta t$-independent, we isolate the temporal order via a successive-difference estimator~\cite{han2025turbulence,olof_convergence} (used for the same purpose in~\cite{CibikFang26}, Remark~4.2),

\begin{equation*}
		p=\log_2\!\frac{\|v_{\Delta t}-v_{\Delta t/2}\|}{\|v_{\Delta t/2}-v_{\Delta t/4}\|},
\end{equation*}
on a fixed mesh ($n_h=32$, $n_H=8$, $T=1.5$). Since the $\mathcal{O}(H)$ contribution is independent of $\Delta t$, it cancels in the estimator, so $p$ measures the temporal order alone. We denote $\Delta v$ and $\Delta p$ the difference of two successive solutions for velocity and pressure, respectively; and $p_v$ as the convergence rate for velocity and $p_p$ as the convergence rate of pressure.

Figure~\ref{fig:temporal} shows the observed temporal convergence rates for the backward Euler scheme. The velocity converges at first order throughout the refinements, in agreement with the theoretical $\mathcal{O}(\Delta t)$ error estimate. The pressure convergence rate approaches first order as $\Delta t$ decreases, indicating that the asymptotic temporal regime has been reached.

\begin{figure}[H]
		\centering
		\begin{minipage}{0.56\linewidth}
	\centering\small
			\begin{tabular}{ccccc}
				\hline
				$\Delta t$ & $\|\Delta v\|$ & $p_v$ & $\|\Delta q\|$ & $p_p$ \\
				\hline
				$0.1$   & $3.36\!\times\!10^{-4}$ & 1.00 & $1.08\!\times\!10^{-5}$ & 1.34 \\
				$0.05$  & $1.68\!\times\!10^{-4}$ & 1.00 & $4.28\!\times\!10^{-6}$ & 1.15 \\
				$0.025$ & $8.41\!\times\!10^{-5}$ & 1.00 & $1.92\!\times\!10^{-6}$ & 1.07 \\
				\hline
			\end{tabular}
		\end{minipage}%
		\begin{minipage}{0.44\linewidth}
			\centering
			\includegraphics[width=\linewidth]{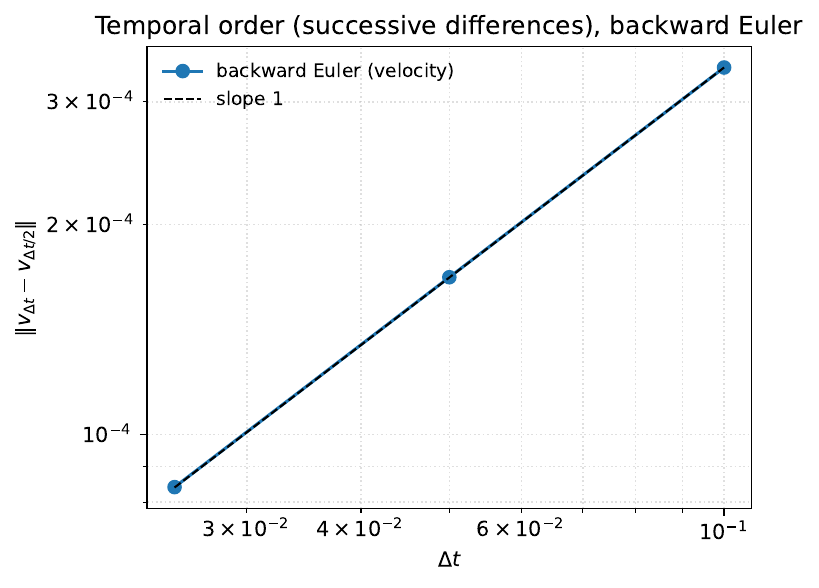}
		\end{minipage}
		\caption{Backward Euler temporal order: $p_v\approx1.00$ at every refinement, confirming the first-order $C_{\Delta t}\Delta t$ term exactly; $p_p$ tightens toward $1$ as $\Delta t$ shrinks.}
		\label{fig:temporal}
\end{figure}
	
\subsection{Effect of the observation mesh size $H$}\label{sec:num-floor}
To examine the effect of the observation mesh size, we apply the $L^2$-projection
$I_H$ to the exact solution to obtain observations on a coarse mesh. We fix the fine mesh $n_h=32$, $\Delta t=0.02$, $T=1.5$, $\chi=100$, $\mu_1=\mu_2=100$, varying the coarse observation mesh $n_H\in\{4,6,8,12\}$. 

Figure~\ref{fig:hfloor} shows the velocity and pressure errors as functions of the observation mesh size $H$. The fitted log–log slopes are $2.00$ for the velocity and $1.16$ for the pressure. These results are consistent with the theoretical $\mathcal{O}(H)$ dependence predicted by the error estimate. The observed second-order convergence of the velocity error indicates that the $\mathcal{O}(H)$ model-error term is not the dominant source of error over the range of $H$ considered.
	
\begin{figure}[H]
	\centering
		\begin{minipage}{0.42\linewidth}
			\centering\small
			\begin{tabular}{cccc}
				\hline
				$n_H$ & $H$ & $\|u(T)-v^N_h\|$ & $\|p(T)-q^N_h\|$ \\
				\hline
				$4$  & $0.250$ & $3.22\!\times\!10^{-2}$ & $9.66\!\times\!10^{-3}$ \\
				$6$  & $0.167$ & $1.34\!\times\!10^{-2}$ & $6.28\!\times\!10^{-3}$ \\
				$8$  & $0.125$ & $7.59\!\times\!10^{-3}$ & $4.25\!\times\!10^{-3}$ \\
				$12$ & $0.083$ & $3.56\!\times\!10^{-3}$ & $2.74\!\times\!10^{-3}$ \\
				\hline
			\end{tabular}
		\end{minipage}
		\hfill
		\begin{minipage}{0.5\linewidth}
			\centering
			\includegraphics[width=\linewidth]{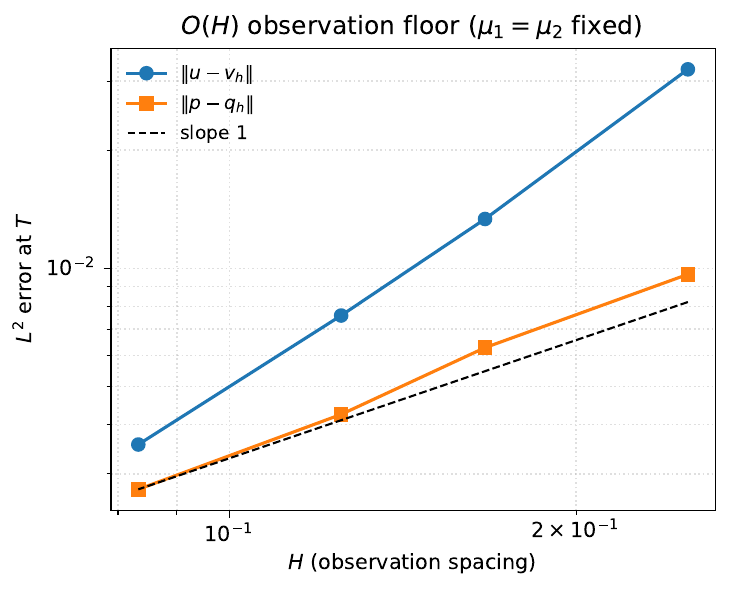}
		\end{minipage}
	\caption{Velocity and pressure errors versus the observation mesh size $H$. The fitted log--log slopes are $2.00$ for the velocity and $1.16$ for the pressure. The pressure slope is consistent with the linear $C_H H$ term in \eqref{err-decomp-num}; the velocity error decays faster than $\mathcal{O}(H)$ over the range of $H$ tested, indicating that the model error term is not the dominant contribution to the velocity error in this regime.}
		\label{fig:hfloor}
	\end{figure}
	
\subsection{Effect of the ratio $\mu_2/\mu_1$ on pressure accuracy}\label{sec:num-mu2}
	
By linearity of $I_H$, we have
    \begin{equation}
        -\mu_1(I_H(p)-I_H(q))-\mu_2(I_H(q)-q) = -\mu_1
	I_H(p)+(\mu_1-\mu_2)I_H(q)+\mu_2 q,
    \end{equation}
and \eqref{estimate-pressure-fully} carries an explicit $(\mu_1-\mu_2)/\mu_1$ term. We vary the ratio $\mu_2/\mu_1$ for two values of the pressure nudging parameter, $\mu_1\in\{16,32\}$, with $n_H=2$ ($H=0.5$), $n_h=32$, $\Delta t=0.02$ and $T=1.5$.  Only the pressure error is reported, since the velocity error is insensitive to $\mu_2$ throughout. The three interior ratios $\mu_2/\mu_1\in\{0.25,0.5,0.75\}$ are included in order to resolve the transition away from $\mu_2=0$.

Figure \ref{fig:musweep} shows the dependence of the pressure error on the ratio $\mu_2 / \mu_1$. The pressure error is largest when no pressure regularization is applied ($\mu_2=0$) and decreases substantially once $\mu_2>0$. The error reaches a shallow minimum near $\mu_2/\mu_1 \approx 1.5-2$ before increasing gradually for larger values of $\mu_2$. Similar trends are observed for both choices of $\mu_1$. These results demonstrate that pressure regularization is important for long-time accuracy.
	
\begin{figure}[H]
	\centering
	\begin{minipage}{0.55\linewidth}
		\centering
		\resizebox{\linewidth}{!}{%
		\begin{tabular}{ccc}
			\hline
			$\mu_2/\mu_1$ & $\|p(T)-q^N_h\|,\mu_1{=}16$ & $\|p(T)-q^N_h\|,\mu_1{=}32$ \\
			\hline
			$0$            & $3.50\!\times\!10^{-1}$ & $3.49\!\times\!10^{-1}$ \\
			$0.25$         & $9.21\!\times\!10^{-2}$ & $5.37\!\times\!10^{-2}$ \\
			$0.5$          & $5.42\!\times\!10^{-2}$ & $3.23\!\times\!10^{-2}$ \\
			$0.75$         & $4.00\!\times\!10^{-2}$ & $2.48\!\times\!10^{-2}$ \\
			$1(\mu_1{=}\mu_2)$ & $3.29\!\times\!10^{-2}$ & $2.13\!\times\!10^{-2}$ \\
			$1.5$          & $2.68\!\times\!10^{-2}$ & $1.89\!\times\!10^{-2}$ \\
			$2$            & $2.59\!\times\!10^{-2}$ & $2.02\!\times\!10^{-2}$ \\
			$3$            & $3.18\!\times\!10^{-2}$ & $2.83\!\times\!10^{-2}$ \\
			$4$            & $4.28\!\times\!10^{-2}$ & $4.00\!\times\!10^{-2}$ \\
			$6$            & $7.18\!\times\!10^{-2}$ & $6.95\!\times\!10^{-2}$ \\
			$8$            & $1.08\!\times\!10^{-1}$ & $1.06\!\times\!10^{-1}$ \\
			\hline
		\end{tabular}%
		}
	\end{minipage}%
	\hfill
	\begin{minipage}{0.43\linewidth}
		\centering
		\includegraphics[width=\linewidth]{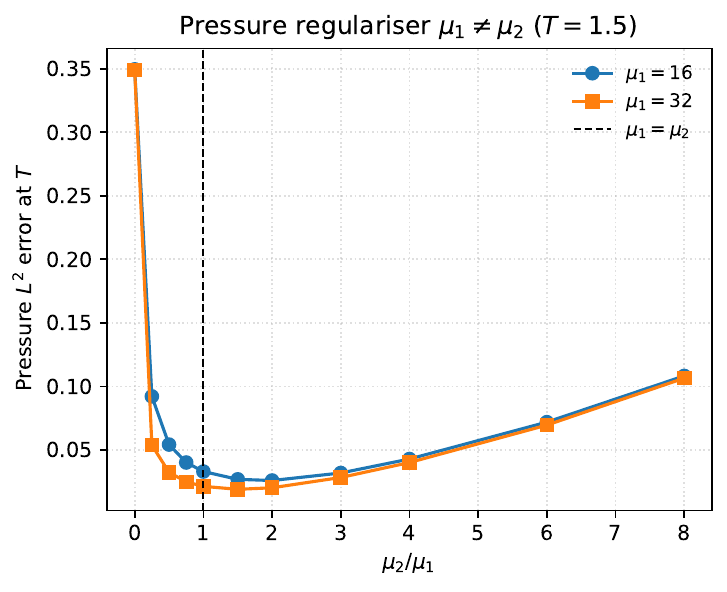}
	\end{minipage}
	\caption{Pressure error versus $\mu_2/\mu_1$ for two values of $\mu_1$. The error is largest at $\mu_2=0$, decreases sharply once a pressure regularizer is introduced, reaches a shallow minimum near $\mu_2/\mu_1\approx1.5$--$2$, and then increases gradually for $\mu_2\gtrsim4\mu_1$. The results are similar for both values of $\mu_1$. The test shows that the regularizer is essential for long-time stability and not merely a minor accuracy correction. Numerics support $\mu_1\ge\mu_2$, with the practical optimum slightly above $\mu_1=\mu_2$.}
	\label{fig:musweep}
\end{figure}
	
\section{Conclusion}\label{sec:discussion}

Section~\ref{error-estimates} establishes semi-discrete finite element error estimate, and a fully discrete error estimate for the backward Euler scheme. The analysis shows that the contributions from the initial velocity and pressure errors decay exponentially fast in time, with decay rate determined by \(\beta=\mathcal{O}(\chi,\mu_1)\). In the backward Euler case, the corresponding discrete decay factor is \((1+\beta\Delta t)^{-1}\) per timestep. The remaining modeling error arises from the mismatch between the slightly compressible NSE and the nudged incompressible system. Under the stated assumptions, the velocity and pressure errors are bounded by terms of size \(\mathcal{O}(H+\mu_1^{-1/2})\).

We verify the theoretical results for the backward Euler scheme with a manufactured exact solution in Section \ref{numerical-tests}.  We use the Taylor--Hood pair, which attains its optimal spatial rates ($\mathcal{O}(h^2)$ pressure, $\mathcal{O}(h^3)$ velocity) until the fixed-$\Delta t$ temporal floor intervenes; balancing $\mathcal{O}(\Delta t)=\mathcal{O}(h^2)$ removes that confound and delivers global second order. In Section \ref{sec:num-floor}, we study the effect of observation data mesh size $H$. The coarser observation mesh gives a clean linear $\mathcal{O}(H)$ pressure floor, which is consistent with the model error of $\mathcal{O}(H+\frac{1}{\sqrt{\mu_1}})$.
    
In our analysis, we assume that $\mu_1 \ge \mu_2$. In Section~\ref{sec:num-mu2}, we examine the pressure accuracy for different choices of $\mu_1$ and $\mu_2$. The numerical results show that omitting $\mu_2$ or taking $\mu_2$ much larger than $\mu_1$ leads to worse pressure accuracy, supporting the theoretical assumption that $\mu_1 \ge \mu_2$. Future work includes analyzing error estimates for higher-order time discretization methods, such as BDF2, and extending the method to other finite element pairs and three--dimensional problems with complex flow structures.

\section*{Acknowledgments}
The authors thank Professor William Layton for suggesting this problem and for many insightful discussions.

\bibliographystyle{plain}
\bibliography{reference_new}
\end{document}